\documentclass[twoside,
]{amsart}

  \usepackage[utf8]{inputenc}
  \usepackage{amsmath}
  \usepackage{amssymb}
  \usepackage{mathtools}
\usepackage{adjustbox} 
  \usepackage{mathrsfs}
  \usepackage{bm}
  \usepackage{bbm}
  \usepackage{pifont}
  \usepackage{graphicx}
  \usepackage{hyperref}
  
  \usepackage[svgnames,x11names]{xcolor} 

  \hypersetup{
    colorlinks, linkcolor={Chocolate4},
    citecolor={Chocolate4}, urlcolor={DarkOliveGreen4}
  }

  \usepackage{color}
  \usepackage{todonotes}

	\newtheoremstyle{slanted}
	{}
	{}
	{\slshape}
	{}
	{\bfseries}
	{.}
	{ }
	{}
	
	\theoremstyle{slanted}
	\newtheorem{theo}{Theorem}[section]
	\newtheorem{prop}[theo]{Proposition}

	\newtheorem{lemma}[theo]{Lemma}
	\newtheorem{definition}[theo]{Definition}
	\newtheorem{corollary}[theo]{Corollary}
	\newtheorem{remark}[theo]{Remark}

	\DeclareMathOperator{\Id}{Id}	
	\newcommand{\tend}[3][]{\xrightarrow[#2\to#3]{#1}}
		
	\newcommand{\EE}{\mathbb{E}}

	\newcommand{\ZZ}{\mathbb{Z}}
	\renewcommand{\AA}{\mathbb{A}}

	\newcommand{\RR}{\mathbb{R}}
	
	\newcommand{\PP}{\mathbb{P}}
	
	\newcommand{\A}{\mathscr{A}}
	\newcommand{\B}{\mathscr{B}}
	\newcommand{\C}{\mathscr{C}}
	\newcommand{\F}{\mathscr{F}}
	\newcommand{\G}{\mathscr{G}}
	
	\renewcommand{\L}{\mathscr{L}}
	\newcommand{\U}{\mathcal{U}}
	\newcommand{\D}{\mathcal{D}}

	\newcommand{\X}{\mathcal{X}}
	
	\newcommand{\perpep}{\rotatebox[origin=c]{90}{$\models$}}
	\newcommand{\perpc}[1]{\stackrel{#1}{\perpep}}

	\newcommand{\eps}{\varepsilon}
	\newcommand{\epst}{\tilde\eps}
	\newcommand{\te}{{\tau_{\eps}}}
	\newcommand{\fte}[1][]{\F_{#1}^{\te}}
	\newcommand{\ftm}[1][]{\F_{#1}^{\tau^{-1}}}
	\newcommand{\tribu}{\boldsymbol{\Sigma}}
	\newcommand{\isoproc}[2]{\stackrel{#1,#2}{\sim}}
	\newcommand{\immersed}{\prec}
	\newcommand{\immersible}{\underset{\sim}{\immersed}}
	\newcommand{\zA}{{0_{\AA}}}
	\newcommand{\cA}{|\AA|}
    \newcommand{\AAe}{\AA_{\text{env}}}
    \newcommand{\fenv}{f_\te^{\text{env}}}
    \newcommand{\Xenv}{X^{\text{env}}}
	\newcommand{\teenv}{{\tau_{\eps}^{\text{env}}}}
	\newcommand{\phiepenv}{{\varphi_\eps^{\text{env}}}}
	\newcommand{\phiep}{{\varphi_\eps}}
	\newcommand{\qm}{{\mathop{?}}}
	
\title[Filtrations from Cellular Automata]{Classification of Backward Filtrations and Factor Filtrations: Examples from Cellular Automata}
\author{Paul Lanthier and Thierry de la Rue}
\address{Laboratoire de Mathématiques Rapha\"el Salem,
Université de Rouen, CNRS,
Avenue de l'Université,
F76801 Saint \'Etienne du Rouvray, France.}
\email{planthier76@outlook.fr, Thierry.de-la-Rue@univ-rouen.fr}
\thanks{Research partially supported by the Regional Research Project Moustic (Normandie)}

\begin{document}
\bibliographystyle{amsplain}

\begin{abstract}
 We consider backward filtrations generated by processes coming from deterministic and probabilistic cellular automata. We prove that these filtrations are standard in the classical sense of Vershik's theory, but we also study them from another point of view that takes into account the measure-preserving action of the shift map, for which each sigma-algebra in the filtrations is invariant. This initiates what we call the dynamical classification of factor filtrations, and the examples we study show that this classification leads to different results.
 
\end{abstract}

\maketitle

\section{Introduction}

The topic of this paper is the study of some backward filtrations, or filtrations in discrete negative time, which are non-decreasing sequences of the form $\F=(\F_n)_{n\leq0}$ where each $\F_n$ is a sub-sigma-algebra on a given probability space. In the sequel, we will simply refer to them as \emph{filtrations}. 

In this work, all measure spaces are implicitely assumed to be Polish spaces equipped with their Borel sigma-algebra, and we only consider random variables taking their values in such spaces. If $X$ is a random variable defined on some probability space $(\Omega,\PP)$, taking values in a Polish space $E$, the \emph{law} of $X$ is the probability measure $\L(X)$ on $E$ which is the pushforward measure of $\PP$ by the measurable map $X$. We call \emph{copy} of $X$ any random variable $X'$, possibly defined on another probability space, such that $\L(X')=\L(X)$ (in particular, $X'$ takes its values in the same space as $X$).
Given a family $(X_m)_{m\in I}$ of random variables defined on a given probability space $(\Omega,\PP)$, we denote by $\tribu(X_m: m\in I)$ the sub-sigma algebra generated by the random variables $X_m$, $m\in I$ and the negligible sets.
All sigma-algebras are supposed to be complete and \emph{essentially separable}: they are generated (modulo negligible sets) by countably many events (or, equivalentely, by countably many random variables). 

Therefore, each filtration $\F=(\F_n)_{n\leq0}$ we consider can be generated by a \emph{process in negative time} (that is to say a family $X=(X_n)_{n\leq0}$ of random variables), which means that for each $n\leq0$, $\F_n=\tribu(X_m: m\leq n)$. The filtration $\F$ is then the mathematical object that describes the acquisition of information as the process $(X_n)_{n\leq0}$ evolves from $-\infty$ up to the present time. 
Different processes can generate the same filtration, 
and the classification of filtrations is roughly equivalent to answering the following question: given a filtration $\F$, can we find a ``nice'' process generating it?

In this direction, the simplest structure we can have is called a \emph{filtration of product type}, which means a filtration which can be generated by a sequence $X=(X_n)_{n\leq0}$ of \emph{independent} random variables. Then there is the slightly more general notion of \emph{standard filtration}: a filtration which is \emph{immersible} into a filtration of product type (see Definition~\ref{def:immersible} below).

Classification of backward filtrations was initiated by Vershik in the 1970's~\cite{Vershik1970}. His work, which was written in the language of ergodic theory, remained quite confidential until a new publication in the 1990's~\cite{vershik1994}. Then some authors started to bring the subject into probability theory where they found nice applications (see in particular~\cite{DFST1996,Tsirelson,Emery-Scha}). 
In the present paper, among other things we propose a return trip to ergodic theory by considering special families of filtrations on measure-theoretic dynamical systems (such systems are given by the action of a measure-preserving transformation $T$ on a probability space $(\Omega,\A,\PP)$). 
Filtrations in measure-theoretic dynamical systems have already been considered by many authors (see \textit{e.g.} \cite{Hoffman2000,HR2002,BMMSM2006,CL2017}), but in situations where the transformation acts ``in the direction of the filtration'': 
by this we mean that, in these works, the filtrations satisfy $\F_{n-1}=T^{-1}\F_n\varsubsetneq \F_n$ for each $n\le 0$. 
Here we adopt a transverse point of view: we consider examples where the sigma-algebras forming the filtration are invariant with respect to the underlying transformation: $\F_n=T^{-1}\F_n$ for each $n\le0$. We get what we call \emph{factor filtrations}, and we would like to classify these objects by taking into account the transverse action of the transformation. 
To emphasize the role of the underlying dynamical system, we will call the classification of factor filtrations the \emph{dynamical} classification of factor filtrations. By opposition, we will refer to the usual classification of filtrations as the \emph{static} classification.

The examples that we consider in this paper are all derived from the theory of cellular automata, where the underlying measure-preserving transformation is the shift of coordinates. We provide in this context natural examples of factor filtrations, some of which show a different behaviour depending on whether we look at them from a static or a dynamical point of view.

\subsection{Examples of filtrations built from cellular automata}

We define in this section the two filtrations which will be studied in the paper. The first one is built from an algebraic deterministic cellular automaton $\tau$, and the second one from a probabilistic cellular automaton (PCA) which is a random perturbation of $\tau$, depending on a parameter $\eps$ and denoted by $\te$. Both automata are one-dimensional (cells are indexed by $\ZZ$) and the state of each cell is an element of a fixed finite Abelian group $(\AA,+)$, which we assume non trivial: $\cA\ge2$. 

Throughout the paper, the index $n$ is viewed as the time during which the cellular automata evolve (represented vertically from top to bottom in the figures), whereas the index $i$ should be interpreted as the position within a given configuration (represented horizontally). For a given configuration $x=\bigl(x(i):i\in\ZZ\bigr)$ and $j\le k\in\ZZ$, we denote by $x[j,k]$ the restriction $\bigl(x(i):j\le i\le k\bigr)$ of $x$ to the sites between $j$ and $k$.

\subsubsection{The deterministic cellular automaton $\tau$ and the filtration $\ftm$ }
\label{sec:ftm}

We first define the deterministic cellular automaton $\tau:\AA^\ZZ\to\AA^\ZZ$. Given a configuration $x=\bigl(x(i)\bigr)_{i\in\ZZ}$, the image configuration $\tau x$ is 
given by 
\begin{equation}
 \label{def:tau}
 \forall i\in\ZZ,\ \tau x(i) := x(i)+x(i+1). 
\end{equation}
Observe that $\tau$ can be written as $\tau=\sigma+\Id$, where $\sigma$ is the left shift on $\AA^\ZZ$. 

For any finite set $S$, we denote by $\U_S$ the uniform probability measure on $S$. We consider on $\AA^\ZZ$ the product probability measure $\mu:=\U_\AA^{\otimes\ZZ}$, which is the normalized Haar measure on the compact Abelian group $\AA^\ZZ$. The measure $\mu$ is invariant by $\tau$.

On the probability space $(\AA^\ZZ,\mu)$, we construct a negative-time process $(X_n)_{n\le 0}$ as follows: we first consider the random variable $X_0$ taking values in $\AA^\ZZ$ and simply defined by the identity on $\AA^\ZZ$. Then we set for each integer $n\le0$,
\[
 X_n := \tau^{|n|} X_0.
\]
Since the law $\mu$ of $X_0$ is preserved by $\tau$, the process $(X_n)_{n\le 0}$ defined in this way is stationary. In particular, all the random variables $X_n$, $n\le0$, have the same law $\mu$. 

We denote by $\ftm$ the filtration generated by this process $(X_n)_{n\le0}$: for each $n\le0$, $\ftm[n]=\tribu(X_m:m\le n)=\tribu(X_m)$. We use $\tau^{-1}$ instead of $\tau$ in this notation to insist on the fact that the time of the process generating this filtration goes in the other direction than that of the cellular automaton $\tau$.

\subsubsection{The probabilistic cellular automaton $\te$ and the filtration $\fte$ }
\label{sec:fte}

We also introduce a random perturbation of the deterministic cellular automaton $\tau$, depending on a parameter $\eps$ which is the probability of making an error in the computation of the new state of a given cell.  More precisely, we fix $\eps>0$ and we define $\te$ as the Markov kernel on $\AA^\ZZ$ given by the following: for each $x\in\AA^\ZZ$, $\te(x,\cdot)$ 
%
is the law of a random variable of the form $\tau x+\xi$, where $\xi$ is a random error, wih law
\[
  \L(\xi) = \bigotimes_{i\in\ZZ}\Bigl((1-\eps)\delta_{0_\AA}+ \frac{\eps}{\cA -1}\sum_{a\in\AA\setminus\{0_\AA\}}\delta_a\Bigr).
\]
In the sequel, we will always assume that
\begin{equation}
 \label{eq:epsilon}
 \eps < \frac{[\AA|-1}{\cA },
\end{equation}
which is equivalent to $1-\eps>\eps/(\cA-1)$. In other words, we give more weight to $0_\AA$. Setting 
\begin{equation}
 \label{eq:epst}
 \epst:=\eps\dfrac{\cA }{\cA -1}<1,
\end{equation}
we can also write the law of the random error $\xi$ as 
\begin{equation}
 \label{eq:law_of_xi}
 \L(\xi) = \bigotimes_{i\in\ZZ}\Bigl((1-\epst)\delta_{0_\AA}+ \frac{\epst}{\cA }\sum_{a\in\AA}\delta_a\Bigr).
\end{equation}

The following lemma shows that  the probability measure $\mu$ on $\AA^\ZZ$ defined in Section~\ref{sec:ftm} is invariant by this Markov kernel (in fact, as we will prove later, it is the only one: see Corollary~\ref{cor:te_ergodic}). 

\begin{lemma}
 \label{lemma:mu_invariante_te}
 Let $X=\bigl(X(i):i\in\ZZ\bigr)$ and $\xi=\bigl(\xi(i):i\in\ZZ\bigr)$ be two independent random variables taking values in $\AA^\ZZ$, where $\L(X)=\mu$. Then $\L(\tau X + \xi)=\mu$.
\end{lemma}
\begin{proof}
We already know that $\L(\tau X)=\mu$. By independence of $\tau X$ and $\xi$, and since $\mu$ is the Haar measure on $\AA^\ZZ$, for $a\in\AA^\ZZ$ we have
\[
 \L(\tau X+\xi | \xi=a) = \L(\tau X + a) =\mu.
\]
\end{proof}

(Note that the proof of the lemma is valid regardless of the law of the error $\xi$, and that it also proves the independence of $\tau X+\xi$ and $\xi$.)

By invariance of $\mu$ under $\te$, we can construct on some probability space $(\Omega,\PP)$ a stationary Markov chain $(X_n)_{n\in\ZZ}$, where for each $n\in\ZZ$, $X_n$ takes its values in $\AA^\ZZ$, the law of $X_n$ is  $\L(X_n)=\mu$, and for each $x\in\AA^\ZZ$, the conditional distribution of $X_{n+1}$ given $X_n=x$ is 
\[ \L(X_{n+1} | X_n=x) = \te(x,\cdot). \]
We will denote by $\fte$ the filtration generated by the negative-time part of such a process: for each $n\le 0$, $\fte[n]=\Sigma(X_m:m\le n)$. If we want to define $\fte$ in a canonical way, we can always assume that in this construction, $\Omega=(\AA^\ZZ)^\ZZ$ and $\PP$ is the law of the stationary Markov chain on $\Omega$.

\section{Usual (static) classification of filtrations}

\subsection{A very short abstract of the theory} 

We recall the main points in the usual theory of classification of filtrations. We refer to~\cite{Emery-Scha,laurent2011} for details. 

We start with the following  definition of isomorphic filtrations, which is of course equivalent to the definition provided in the cited references.

\begin{definition}[Isomorphism of filtrations]
 \label{def:isomorphism}
 Let $\F$ and $\F'$  be two filtrations, possibly defined on two different probability spaces. We say that they are \emph{isomorphic} if we can find two processes $X=(X_n)_{n\leq0}$ and $X'=(X'_n)_{n\leq0}$ with the same law generating respectively $\F$ and $\F'$. We write in this case $\F\sim\F'$, or $\F\isoproc{X}{X'}\F'$ to specify the processes involved in the isomorphism. 
\end{definition}

It may not seem obvious in the above formulation that this notion of isomorphism is transitive. We refer the reader who would like details on this point to~\cite[Section~1.2.1]{LanthierPhD}.

Let $\F$ and $\F'$ be two isomorphic filtrations, and let $X$ and $X'$ be two processes such that $\F\isoproc{X}{X'}\F'$. If $Y$ is an $\F_0$-measurable random variable, it is a classical result in measure theory that there exists some measurable map $\phi$ satisfying $Y=\phi(X)$ a.s. Then the isomorphism implicitely provides an $\F'_0$-measurable copy of $Y$, namely the random variable $Y':=\phi(X')$. Of course, this copy depends on the choice of the processes $X$ and $X'$ used in the isomorphism, but once these processes are fixed, the copy $Y'$ of $Y$ is almost surely unique (because the map $\phi$ is almost surely unique with respect to the law of $X$).

\medskip

Let $\A$, $\B$ and $\C$ be 3 sub-sigma-algebras on the same probability space, with $\C\subset\A\cap\B$. We recall that $\A$ and $\B$ are \emph{independent conditionally to $\C$} if, whenever $X$ and $Y$ are bounded real random variables, respectively $\A$ and $\B$ measurable, we have 
\[
  \EE\bigl[XY | \C\bigr] = \EE\bigl[X | \C\bigr]\, \EE\bigl[Y | \C\bigr] \quad(\text{a.s.})
\]
We note in this case 
\[ \A \perpc{\C} \B. \]

\begin{definition}[Immersion]
\label{def:immersion}
 Let $\F$ and $\G$ be two filtrations on the same probability space. 
 We say that $\F$ is \emph{immersed} in $\G$ if the following conditions are satisfied:
 \begin{itemize}
  \item for each $n\le 0$, $\F_n\subset \G_n$ ($\F$ is included in $\G$);
  \item for each $n\le -1$, $\F_{n+1}\perpc{\F_n} \G_n$.
 \end{itemize}
We write in this case $\F\immersed\G$. 

We say that $\F$ and $\G$ are \emph{jointly immersed} if both $\F$ and $\G$ are immersed in the filtration $\F\vee\G$.
\end{definition}

\begin{definition}[Immersibility]
\label{def:immersible}
 Let $\F$ and $\G$ be two filtrations, possibly defined on different probability spaces. 
 We say that $\F$ is \emph{immersible} in $\G$ if there exists a filtration $\F'$ immersed in $\G$ such that $\F'$ and $\F$ are isomorphic.
We write in this case $\F\immersible\G$. 
\end{definition}

We can now define the following basic hierarchy of properties for filtrations.

\begin{definition}
 The filtration $\F$ is said to be \emph{of product type} if it can be generated by a process formed of independent random variables.
 
 The filtration $\F$ is said to be \emph{standard} if it is immersible in a filtration of product type.
 
 The filtration $\F$ is said to be \emph{Kolmogorovian} if its tail sigma-algebra
 \[ \F_{-\infty} := \bigcap_{n\le 0}\F_n \]
 is trivial (it only contains events of probability 0 or 1).
\end{definition}

It is a direct consequence of the definitions and of Kolmogorov 0-1 law that the following chain of implications holds:
\[ \text{Product type} \Longrightarrow \text{Standard} \Longrightarrow \text{Kolmogorovian}.\]
A simple example of a standard filtration which is not of product type has been provided by Vinokurov
(see~\cite{Emery-Scha}). The construcion of a Kolmogorovian filtration which is not standard was one of the first spectacular achievements of Vershik in this theory.

\subsection{Filtrations of Product type: Example of $\ftm$}

We now come back to the filtration $\ftm$ defined in Section~\ref{sec:ftm}, and we use here the same notations as in this section. In particular, the process $X=(X_n)_{n\le 0}$ generating $\ftm$ is a process where each coordinate $X_n$ takes its values in $\AA^\ZZ$, $\L(X_n)=\mu$, and $X_{n-1}=\tau X_n$. 

The purpose of this section is to prove the following result.

\begin{theo}
 \label{thm:ftm_product_type}
 The filtration $\ftm$ is of product type.
\end{theo}

The above theorem is a direct consequence of the two following lemmas.

\begin{lemma}
\label{lemma:Xn_in}
  For each sequence $(i_n)_{n\le 0}$ of integers, the random variables $X_n(i_n)$, $n\le 0$, are independent.
\end{lemma}

\begin{lemma}
  \label{lemma:generating_ftm}
  Consider the sequence $(i_n)_{n\le 0}$ of integers defined by $i_n:=-\lfloor |n|/2 \rfloor$. Then the process $\bigl(X_n(i_n)\bigr)_{n\le 0}$ generates the filtration $\ftm$.
\end{lemma}

\begin{proof}[Proof of Lemma~\ref{lemma:Xn_in}]
  Let us first consider, on a probability space $(\Omega,\PP)$, a random variable $Y=\bigl(Y(i):i\in\ZZ\bigr)$ taking values in $\AA^\ZZ$, and such that $\L(Y)=\mu$. Let us fix some integers $i,j,k$ such that $j\le k$ and $i\in\{j,\ldots,k,k+1\}$. For each block $w\in\AA^{\{j,\ldots,k\}}$ and each $a\in\AA$, it is straightforward to check from the construction of $\tau$ that there exists a unique block $w'\in\AA^{\{j,\ldots,k,k+1\}}$ such that
 \parbox[t]{\textwidth}{%
 \begin{itemize}
   \item $w'(i)=a$;
   \item $\bigl( w'(h)+w'(h+1):j\le h\le k\bigr)=w$.
  \end{itemize}
  }
  Let us denote by $\tau_{a,i}^{-1}(w)$ this block $w'$. We then have the following equivalence:
\[
 \tau Y[j,k]=w \Longleftrightarrow \exists a\in\AA,\ Y[j,k+1]=\tau_{a,i}^{-1}(w).
\]
But by construction of $\mu$, all the $\cA $ events $\bigl(Y[j,k+1]=\tau_{a,i}^{-1}(w)\bigr)$, $a\in\AA$ have the same probability $\cA^{-(k-j+2)}$ (and are of course disjoint). It follows that, for each $a\in\AA$, 
\[
 \PP\bigl(Y(i)=a\,|\,\tau Y[j,k]=w\bigr) = \PP\Bigl(Y[j,k+1]=\tau_{a,i}^{-1}(w)\,|\,\tau Y[j,k]=w\Bigr) = \frac{1}{\cA }.
\]
Since this holds for any block $w\in\AA^{\{j,\ldots,k\}}$, this proves that $Y(i)$ is independent of $\tau Y[j,k]$. Then, since this is true for each $j,k$ such that $i\in\{j,\ldots,k,k+1\}$, $Y(i)$ is independent of $\tau Y$.

Let us apply this with $Y=X_n$ for some $n\le 0$: we get that, for each $i_n\in\ZZ$, $X_n(i_n)$ is independent of $\tau X_n=X_{n-1}$. Now if we have an arbitrary sequence $(i_n)_{n\le 0}$ of integers, we note that for each $n\ge 0$, the random variables $X_m(i_m)$, $m\le n-1$ are all $\ftm[n-1]$-measurable, hence $X_n(i_n)$ is independent of $\tribu\bigl(X_m(i_m):m\le n-1\bigr)$.
\end{proof}

Lemma~\ref{lemma:generating_ftm} will be derived from the following result (see Figure~\ref{fig:chemin}).

\begin{figure}[h]
 \includegraphics[width=10cm]{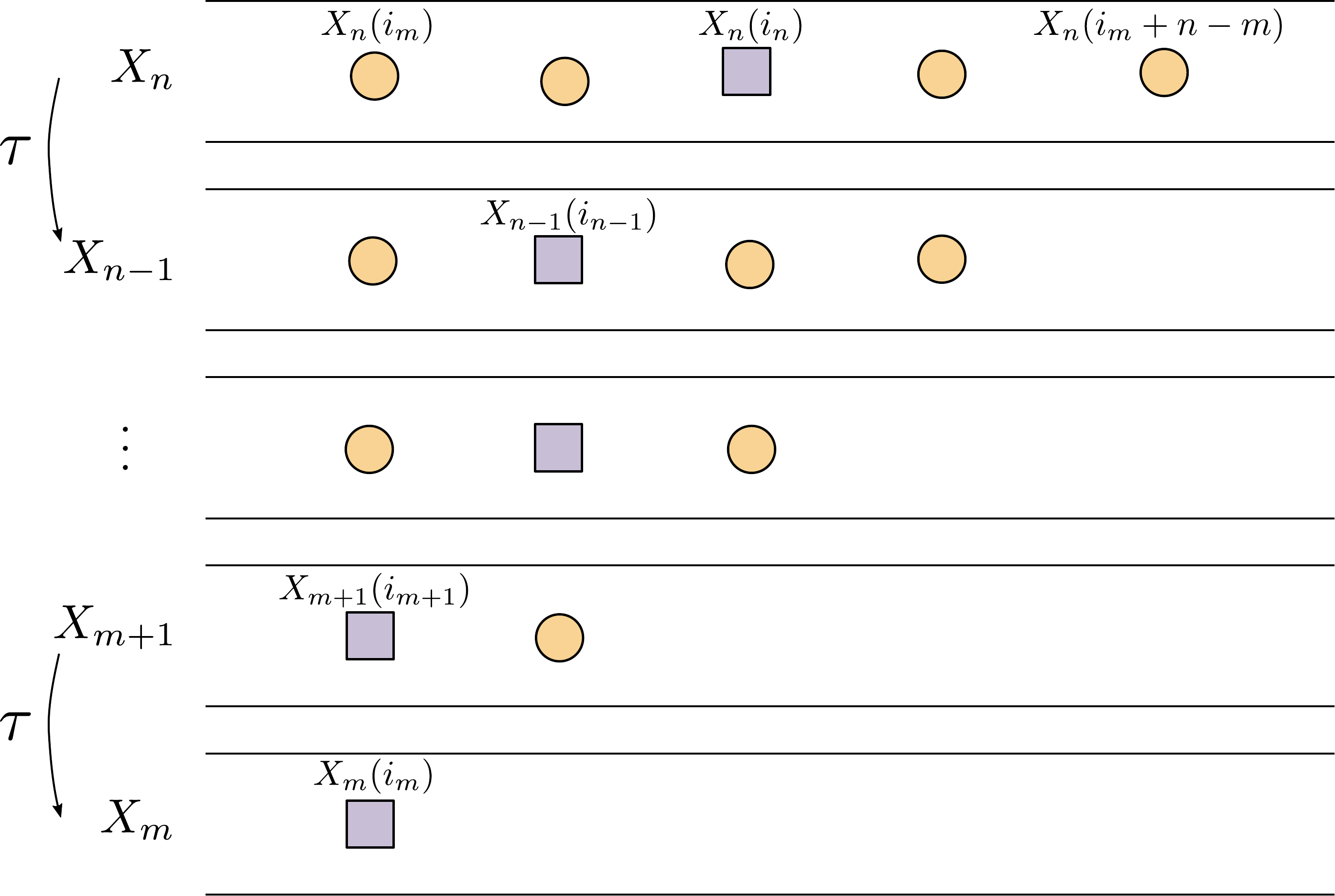}
\caption{Illustration of Lemma~\ref{lemma:chemin}: the variables marked with a circle can be computed from the ones marked with a square.}
\label{fig:chemin}
\end{figure} 

\begin{lemma}
\label{lemma:chemin}
Consider integers $m \leq n  \leq 0$, and let $(i_m,i_{m+1},...,i_n)$ be a finite sequence of integers such that, for each $m \leq \ell \leq n-1$, 
\begin{equation}
 \label{eq:condition_in}
 i_{\ell+1} = i_{\ell} \quad\text{or}\quad i_{\ell+1} = i_{\ell}+1.
\end{equation} 
Then $X_n[i_m,i_m+n-m]$ is measurable with respect to $\Sigma\bigl( X_m(i_m),..., X_n(i_n)\bigr)$.
\end{lemma}

\begin{proof}
For a fixed $m$, we prove the result by induction on $n$. If $n=m$ the result is obvious. Assume now that, for some $m+1\le n\le 0$, the result holds up to $n-1$. Then 
$X_{n-1}[i_m,i_m+n-1-m]$ is measurable with respect to $\Sigma\bigl( X_m(i_m),..., X_{n-1}(i_{n-1})\bigr)$. Remembering the notation $\tau_{a,i}^{-1}(w)$ from the proof of Lemma~\ref{lemma:Xn_in}, and since by assumption we have $i_n\in \{i_m,\ldots,i_m+n-m\}$, we can then write
\[
   X_n[i_m,i_m+n-m] = \tau_{X_n(i_n),i_n}^{-1}\bigl(X_{n-1}[i_m,i_m+n-1-m]\bigr),
\]
hence $X_n[i_m,i_m+n-m]$ is measurable with respect to $\Sigma\bigl( X_m(i_m),..., X_n(i_n)\bigr)$ as claimed.
\end{proof}

\begin{proof}[Proof of Lemma~\ref{lemma:generating_ftm}]
 The sequence $(i_n)$ defined in the statement of Lemma~\ref{lemma:generating_ftm} obviously satisfies~\eqref{eq:condition_in}. Moreover, we have $i_m\to -\infty$ as $m\to -\infty$, and for any fixed $n\le 0$, $i_m+n-m\to +\infty$ as $m\to -\infty$. By application of Lemma~\ref{lemma:chemin}, it follows that $X_n$ is measurable with respect to $\tribu\bigl(X_m(i_m):m\le n\bigr)$. But conversely, all the random variables $X_m(i_m)$, $m\le n$, are $\ftm[n]$-measurable. Hence these variables generate $\ftm[n]$.
\end{proof}

\subsection{Standardness and I-cosiness}

The example of the filtration $\ftm$ is very special, as it is not so hard to explicit a process with independent coordinates which generates the filtration. In general, when we have a standard filtration, it may be very hard to find such a process from which the filtration is built. 
This is one of the reasons why several criteria of standardness have been developped, allowing to prove the standardness of a filtration without giving explicitely the process with independent coordinates. The first such criterion was given by Vershik~\cite{vershik1994}, but here we will be interested in another one, called \emph{I-cosiness}, introduced by Émery and Schachermayer~\cite{Emery-Scha} and strongly inspired by ideas of Tsirelson~\cite{Tsirelson} and Smorodinsky~\cite{Smoro}.

We first have to define the concept of \emph{real-time coupling} for a filtration.

\begin{definition}[Real-time coupling]
\label{def:rtc}
 Let $\F=(\F_n)_{n\le 0}$ a filtration on a probability space $(\Omega,\PP)$. We call \emph{real-time coupling of $\F$} a pair $(\F',\F'')$ of filtrations, both defined on the same probability space (but possibly different from $(\Omega,\PP)$), such that
 \begin{itemize}
  \item $\F'\sim\F$,
  \item $\F''\sim\F$,
  \item $\F'$ and $\F''$ are jointly immersed.
 \end{itemize}
Such a real-time coupling of $\F$ is said to be \emph{independent in the distant past} if there exists some integer $n_0\le0$ such that $\F'_{n_0}$ and $\F''_{n_0}$ are independent. In this case, we also say \emph{$n_0$-independent} if we want to highlight $n_0$.
\end{definition}

In practice, if we want to construct a real-time coupling of a filtration generated by a process $X=(X_n)_{n\le0}$, we have to build on the same probability space two copies $X'=(X'_n)_{n\le0}$ and $X''=(X''_n)_{n\le0}$ of $X$. The joint immersion of the filtrations they generate amonts to the following conditions for each $n\le -1$: 
\begin{align*}
 &\L\bigl(X'_{n+1} | \F'_n \vee \F''_n\bigr) = \L\bigl(X'_{n+1} | \F'_n\bigr),\\
 \text{and}\quad &\L\bigl(X''_{n+1} | \F'_n \vee \F''_n\bigr) = \L\bigl(X''_{n+1} | \F''_n\bigr).
\end{align*}
So, we can construct such a coupling step by step: assuming that we already have defined $X'_m$ and $X''_m$ for each $m\le n$, the construction continues at time $n+1$ with the realization, conditionally to $\F'_n\vee\F''_n$, of a coupling of the two conditional laws  $\L\bigl(X'_{n+1} | \F'_n\bigr)$ and $\L\bigl(X''_{n+1} | \F''_n\bigr)$. This explains the denomination \emph{real-time coupling}.

\begin{definition}[I-cosiness]
 \label{def:I-cosiness}
 Let $\F$ be a filtration, and let $Y$ be an $\F_0$-measurable random variable taking values in some Polish metric space $(E,d)$. 
 Then $Y$ is said to be \emph{I-cosy (with respect to $\F$)} if, for each real number $\delta>0$, we can find a real-time coupling $(\F',\F'')$ of $\F$ which is independent in the distant past, and such that the two copies $Y'$ and $Y''$ of $Y$ in $\F'$ and $\F''$ respectively satisfy
 \begin{equation}
  \label{eq:close_copies}
  d(Y',Y'')<\delta \text{ with probability } > 1-\delta.
 \end{equation}
 The filtration itself is said to be \emph{I-cosy} if each $\F_0$-measurable random variable $Y$ is I-cosy with respect to $\F$.
\end{definition}

The variable $Y$ whose copies we want to be close together is called a \emph{target}. Classical arguments of measure theory allow to considerably reduce the number of targets to test when we want to establish I-cosiness for a given filtration (see~\cite{laurent2011}, or~\cite[Section~1.2.4]{LanthierPhD}). In particular we will use the following proposition.

\begin{prop}
 \label{prop:targets}
 Assume that $\F_0$ is generated by a countable family $(Z_i)_{i\in I}$ of random variables taking values in finite sets. Assume also that we have written the countable set $I$ as an increasing union of finite sets $I_k$, $k\ge0$. For each integer $k\ge0$, denote by $Y_k$ the random variable $(Z_i)_{i\in I_k}$. If $Y_k$ is I-cosy for each $k$, then the filtration is I-cosy.
\end{prop}
 
Note that, in the context of the above proposition, since the set of all possible values of $Y_k$ is finite, we can endow it with the discrete metric and replace condition~\eqref{eq:close_copies} with 
\[
  Y'_k=Y''_k \text{ with probability } > 1-\delta.
\]

The following result is proved in~\cite{Emery-Scha} using ideas from~\cite{vershik1994} (see also~\cite{laurent2011}).

\begin{theo}
 \label{thm:criterion}
 A filtration $\F$ is standard if and only if it is I-cosy.
\end{theo}

\subsection{Standardness of the filtration $\fte$}
\label{sec:fte_standard}
We will apply the I-cosiness criterion to prove the standardness of the filtration $\fte$ defined in Section~\ref{sec:ftm}, from which we will also be able to derive that this filtration is of product type. We use now the notations introduced in this section: $X=(X_n)_{n\in\ZZ}$ is a stationary Markov process with transitions given by the Markov kernel $\te$, for each $n$, $X_n$ takes its values in $\AA^\ZZ$ and follows the law $\mu$. The filtration $\fte$ is generated by the negative-time part of this process. 

\begin{theo}
 \label{thm:fte_standard}
 For each $0 < \eps < \frac{[\AA|-1}{\cA }$, the filtration $\fte$ is I-cosy, hence it is standard.
\end{theo}

\begin{proof}
  Since $\fte[0]$ is generated by the countable family of random variables $\bigl(X_n(i):n\leq0,i\in\ZZ \bigr)$, it is sufficient by Proposition~\ref{prop:targets} to check that, for each integer $k\ge0$, the random variable $Y_k:=\bigl(X_n[-k,k]:-k\le n\leq0 \bigr)$ is I-cosy. In fact, we will see at the end of the proof that it is enough by stationarity of the process $X$ to consider simpler targets, which are the random variables of the form $X_0[-k,k]$, $(k\ge0)$.
  
  So we fix an integer $k\ge0$, we consider the target $X_0[-k,k]$, and we fix a real number $\delta>0$. To check the I-cosiness of $X_0[-k,k]$, we have to construct on some probability space $(\Omega,\PP)$ two copies $X'$ and $X''$ of the process $X$, such that 
  \begin{itemize}
  \item for some $n_0\le 0$,
  \begin{equation}
   \label{eq:n_0-independence}
   \F'_{n_0}=\tribu(X'_m:m\le n_0)\text{ and }\F''_{n_0}=\tribu(X''_m:m\le n_0)\text{ are independent, }
  \end{equation}
(to ensure that the filtrations $\F'$ and $\F''$ generated by the negative parts of these process are independent in the distant past),
\item for each $n_0+1\le n\le 0$,
\begin{equation}
 \label{eq:joint_immersion}
 \L(X'_{n} | \F'_{n-1}\vee\F''_{n-1}) = \te(X'_{n-1}, \cdot) \text{ and }\L(X''_{n} | \F'_{n-1}\vee\F''_{n-1}) = \te(X''_{n-1}, \cdot)
\end{equation}
(to get the joint immersion of the filtrations $\F'$ and $\F''$), and
\item the copies $\bigl(X'_0[-k,k]\bigr)$ and $\bigl(X''_0[-k,k]\bigr)$ of the target random variable satisfy
\begin{equation}
 \label{eq:Tk}
 \PP\bigl(X'_0[-k,k] = X''_0[-k,k]\bigr) > 1-\delta.
\end{equation}
\end{itemize}
Here is how we will proceed. We fix some $n_0\le 0$ and consider a probability space $(\Omega,\PP)$ in which we have two independent copies $\bigl(X'_n:n\le n_0\bigr)$ and $\bigl(X''_n:n\le n_0\bigr)$ of $\bigl(X_n:n\le n_0\bigr)$, and an independent family of i.i.d random variables $\bigl(U_n(i):n\ge n_0+1,i\in\ZZ\bigr)$ which are all uniformly distributed on the interval $[0,1]$. These random variables will be used to construct inductively the error processes $\bigl(\xi'_n(i):n\ge n_0+1,i\in\ZZ\bigr)$ and $\bigl(\xi''_n(i):n\ge n_0+1,i\in\ZZ\bigr)$ and the rest of the Markov processes $\bigl(X'_n:n\ge n_0+1\bigr)$ and $\bigl(X''_n:n\ge n_0+1\bigr)$ through the formula
\[
 X'_{n}:=\tau X'_{n-1} + \xi'_{n},\text{ and } X''_{n}:=\tau X''_{n-1} + \xi''_{n}. 
\]
We will use the auxilliary process $Z:=X'-X''$, and we note that for $n\ge n_0+1$ we have 
\begin{equation}
 \label{eq:evolution_Zn}
 Z_n=\tau Z_{n-1}+\xi'_n-\xi''_n.
\end{equation}
Rewriting~\eqref{eq:Tk} with this notation, we want to achieve the coupling in such a way that, provided $|n_0|$ is large enough,
\begin{equation}
 \label{eq:goal}
 \PP\bigl(Z_0[-k,k] = (0_\AA,\ldots,0_\AA)\bigr) > 1-\delta. 
\end{equation}
We explain now how we construct inductively the error processes. Assuming that, for some $n\ge n_0+1$, we already have defined the processes $X'$ and $X''$ up to time $n-1$, we choose for each $i\in\ZZ$ two random maps $g'_{n,i}$ and $g''_{n_i}$ from $[0,1]$ to $\AA$ and set $\xi'_{n}(i):=g'_{n,i}\bigl(U_n(i)\bigr)$ and $\xi''_{n}(i):=g''_{n,i}\bigl(U_n(i)\bigr)$. These maps are random in the sense that they depend on the realizations of the processes up to time $n-1$. However they have to comply with the required law for the errors, which we ensure by choosing them among the family of maps $(g_a)_{a\in\AA}$ defined as follows. Recalling the definition~\eqref{eq:epst} of $\epst$ and the formulation~\eqref{eq:law_of_xi} of the law of the error, we cut the interval $[0,1]$ into $\cA +1$ subintervals: the first $\cA$ of them are of the same length $\epst/\cA$, we denote them by $J_b$, $b\in\AA$ (they correspond to the uniform measure part of the error), and the last one is  $[\epst,1]$. Then for each $a\in\AA$, we define the map $g_a:[0,1]\to\AA$ by setting 
\[ g_a(u):= \begin{cases}
             a+b &\text{ if }u\in J_b\ (b\in\AA),\\
             \zA &\text{ if }u\in [\eps,1].
            \end{cases}
\]
By choosing $g'_{n,i}$ and $g''_{n_i}$ in the set $\{g_a: a\in\AA\}$, we get the correct conditionnal law of the errors $\xi'_n$ and $\xi''_n$ knowing $\F'_{n-1}\vee\F''_{n-1}$, and we have~\eqref{eq:joint_immersion}. It remains to explain how we make this choice for our purposes.

Given a site $(n,i)$, for each $a\in\AA$ we define the strategy $S_a$ as the choice $\bigl(g'_{n,i},g''_{n,i}\bigr):=(g_a,g_{\zA})$ (see Figure~\ref{fig:Sa}). 
In this way, when we apply the strategy  $S_a$, we obtain 
\[
 \xi'_n(i)-\xi''_n(i)=\begin{cases}
                        a &\text { if }U_n(i)\in[0,\epst),\\
                        \zA &\text{ otherwise.}
                       \end{cases}
\]
\begin{figure}
\begin{center}
\includegraphics[width=8cm]{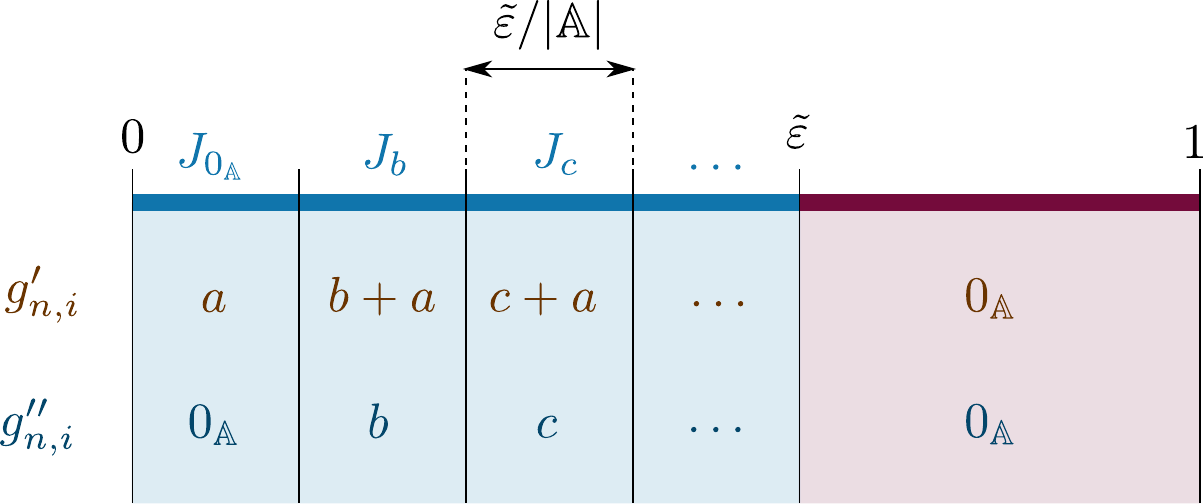}		                                 
\end{center}
\caption{Choice of $g'_{n,i}$ and $g''_{n,i}$ when we apply the strategy $S_a$}
	\label{fig:Sa}
\end{figure} 

The special case $a=\zA$ gives $\xi''_n(i)-\xi'_n(i)=\zA$ with probability one, and thus the choice of the strategy $S_\zA$ on a given site $(n,i)$ ensures that, locally, the evolution of the $Z$ process is given by the  the determinist action of the cellular automaton $\tau$ (remember~\eqref{eq:evolution_Zn}).

The set of sites $(n,i)$ for which we have to define $g'_{n,i}$ and $g''_{n,i}$ is partitionned into ``diagonals'' ${\D} (j)$, $j\in\ZZ$, where 
\[ {\D} (j) := \{ (n,i) / n\geq n_0+1,\ i=j-n\}. \]
(See Figure~\ref{fig:CTR-tau_epsilon}.) We observe that, for $i<j\in\ZZ$, any error added on some site of 
 ${\D} (j)$ has no influence on $Z_0(i)$. 

\begin{figure}
 \begin{center}
  \includegraphics[width=135mm]{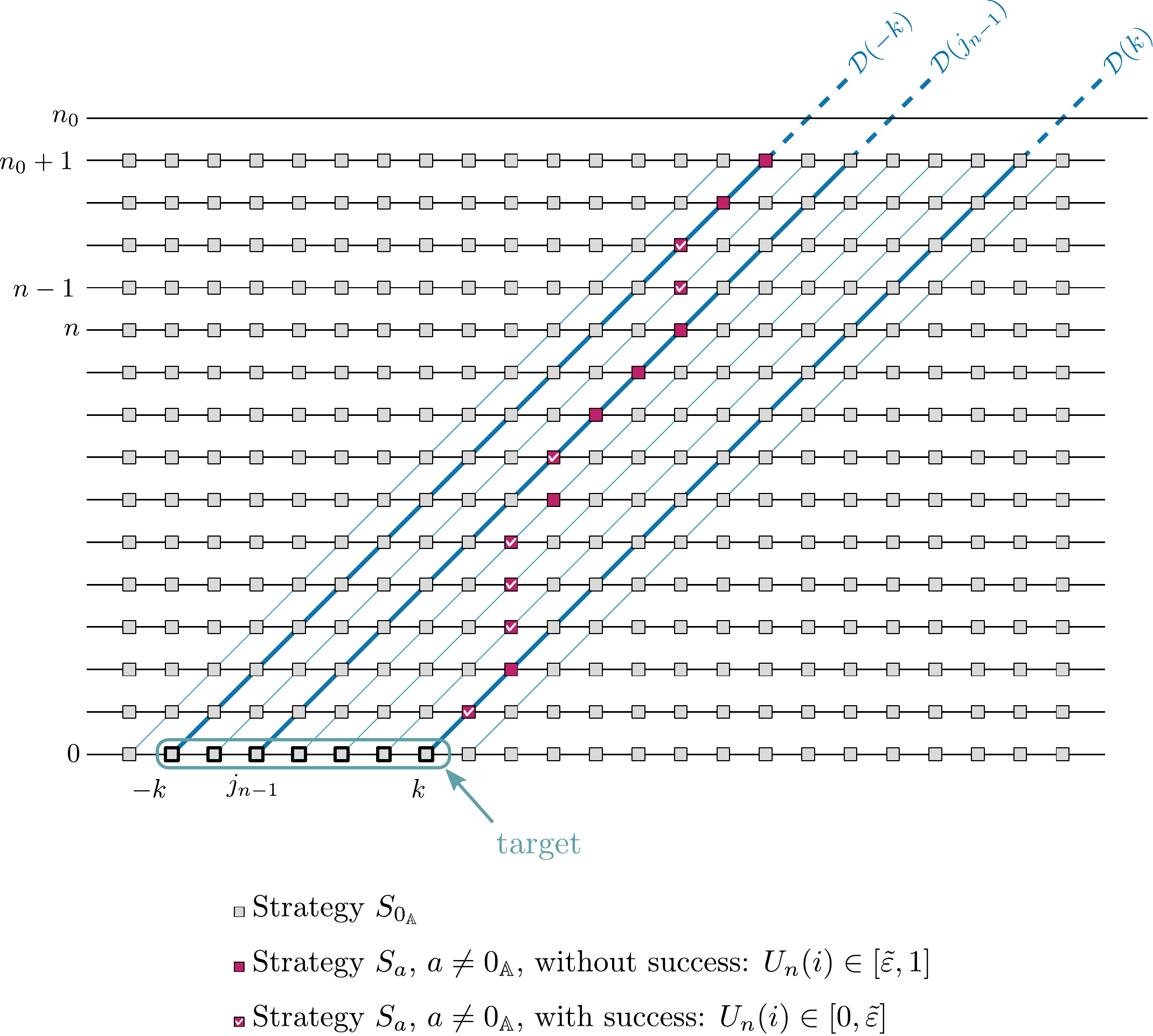}
 \end{center}
\caption{The choice of the local strategies to construct the real-time coupling}
\label{fig:CTR-tau_epsilon}
\end{figure} 

On all the diagonals ${\D} (j)$, $j<-k$ or $j>k$, we systematically choose the strategy $S_\zA$, so that $Z$ follows the deterministic evolution of the cellular automaton $\tau$ in these areas. It remains to explain the choice of the strategies on the sites of the diagonals  ${\D} (j)$, $-k\leq j\leq k$. 

Consider an integer $n$, $n_0<n\leq 0$, and assume that we know the $Z$ process up to time $n-1$. We compute $\tau^{|n-1|}(Z_{n-1})$, and according to what we get we make the following choices:
\begin{itemize}
 \item If $\tau^{|n-1|}(Z_{n-1})[-k,k]=(\zA,\ldots,\zA)$, we can easily ensure that $Z_0[-k,k]=(\zA,\ldots,\zA)$ by choosing the strategy  $S_{\zA}$ for all the following sites. We set in this case $j_{n-1}:=k+1$.
 \item Otherwise, we define $j_{n-1}$ as the smallest integer $j\in\{-k,\ldots,k\}$ such that 
 \[ \tau^{|n-1|}(Z_{n-1})(j)\neq \zA, \] 
 and we apply the strategy $S_{\zA}$ at all sites $(n,i)$ except when $(n,i)$ falls on $\D_{j_{n-1}}$ (that is, when $i=j_{n-1}-n$), where we apply $S_{-a_{n-1}}$ with $a_{n-1}:= \tau^{|n-1|}(Z_{n-1})(j_{n-1})$. 
\end{itemize}
With this method, if $j_{n-1}\in\{-k,\ldots,k\}$, the only $i$ on line $n$ for which $Z_{n,i}$ is not determined is  $i=j_{n-1}-n$, and the value of $Z_{n,j_{n-1}-n}$ now depends only on the random variable $U_n(j_{n-1}-n)$:
\begin{itemize}
 \item If $U_n(j_{n-1}-n) \in [0,\epst)$, the difference $\xi'_n(j_{n-1}-n)-\xi''_n(j_{n-1}-n)$ takes the value $-a_{n-1}$. This exactly compensates the determinist action of $\tau$ for the target site $(0,j_{n-1})$, which ensures that, at the next step, we get $\tau^{|n|}(Z_{n})(j_{n-1})=\zA$, and thus $j_n\ge j_{n-1}+1$. 
 \item Otherwise, we have $\xi'_n(j_{n-1}-n)-\xi''_n(j_{n-1}-n)=0$, so that $\tau^{|n|}(Z_{n})(j_{n-1})=a_{n-1}$ and $j_n=j_{n-1}$. 
\end{itemize}

We construct inductively the real-time coupling by applying the above method for $n_0+1\le n\le 0$. How we build the coupling for times $n\geq0$ is irrelevant, we can for example choose always the strategy $S_{\zA}$ for these positive times. 
Now we have to prove that \eqref{eq:goal} is achieved for $|n_0|$ large enough.  

We define inductively random times $T(-k),\ldots,T(k)$. We start with $T(-k)$, which only depends on the random variables $U_n(i)$, $(n,i)\in{\D} (-k)$ :
\[ T(-k):=\min\bigl\{n \geq n_0+1  :\, U_n(-k-n) \in [0,\epst) \bigr\}. \]
(With the usual convention that $\min \emptyset=+\infty$, but observe that $T(-k)<\infty$ with probability 1.) The random variable $T(-k)-n_0$ then follows a geometric law of parameter $\epst$. And by construction of the coupling, for each $n_0+1\le n\le 0$ we have 
\[ T(-k)\leq n \Longrightarrow j_n > -k. \]
Then, if we have already defined $T(j)$ for some $-k\le j\le k-1$, we set 
\[ T(j+1):=\min\bigl\{n\geq T(j)+1 :\, U_n(j+1-n) \in [0,\epst) \bigr\}. \]

Note that $T(j+1)$ only depends on $T(j)$ and on the random variables $U_n(i)$, $(n,i)\in{\D} (j+1)$. And by construction of the coupling, we get by induction that, for each $j\in\{-k,\ldots,k\}$ and each $n\leq0$, 
$$ T(j)\leq n \Longrightarrow j_n > j. $$
In particular,
\begin{equation}
 \label{eq:cible_atteinte}
 T(k)\leq 0 \Longrightarrow j_0 > k \Longrightarrow Z_0[-k,k]=(\zA,\ldots,\zA).
\end{equation}

For $j\in\{-k,\ldots,k-1\}$, the difference $T(j+1) -T(j)$ is distributed according to a geometric law of parameter  $\epst$, and is independent of $T(-k),...,T(j)$. Hence $T(k)-n_0$ is distributed as the sum of $2k+1$ independent geometric random variables of parameter $\epst$. Therefore, as $\epst$ and $k$ are fixed, we get
\begin{equation}
 \label{eq:somme_geometrique}
 \PP\bigl(T(k)\leq 0\bigr) \tend{n_0}{-\infty} 1.
\end{equation}
By~\eqref{eq:somme_geometrique} and~\eqref{eq:cible_atteinte}, we get that~\eqref{eq:goal} is satisfied for $|n_0|$ large enough. 

\medskip

This proves that for each $k\ge0$, the target $X_0[-k,k]$ is I-cosy. Now, the same argument works also to prove the I-cosiness of $X_{-k}[-2k,2k]$ for each $k\ge0$.
This gives us a real-time coupling up to time $-k$ for which 
\[ \PP\bigl(Z_{-k}[-2k,2k]=(\zA,\ldots,\zA)\bigr)>1-\delta. \] 
Note that 
\[Z_{-k}[-2k,2k]=(\zA,\ldots,\zA) \Longrightarrow \forall 0\le \ell\le k,\ \tau^\ell Z_k[-k,k]=(\zA,\ldots,\zA)\]
So we can continue this coupling by always using the strategy $S_{\zA}$ on lines $-k+1,\ldots,0$ to prove the I-cosiness of $Y_k=\bigl(X_n[-k,k]:-k\le n\leq0 \bigr)$.
\end{proof}

\medskip

We can furthermore observe that the filtration $\fte$ is \emph{homogeneous}: for each $n\le0$, $\fte[n]=\fte[n-1]\vee\tribu(\xi_n)$ where 
\begin{itemize}
 \item $\xi_n$ is independent of $\fte[n-1]$,
 \item the law of $\xi_n$ is diffuse.
\end{itemize}
Then a direct application of Theorem~A in~\cite{laurent2011} yields the following corollary:

\begin{corollary}
 \label{cor:product_type}
 For each $0 < \eps < \frac{[\AA|-1}{\cA }$, the filtration $\fte$ is of product type.
\end{corollary}

\subsection{Uniform cosiness and ergodicity of the Markov kernel}
        
   If we look carefully at the proof that the filtration $\fte$ is I-cosy, we see that the probability that the two copies of the target coincide in the real-time coupling converge to 1 as       
   $|n_0|\to\infty$, u\emph{niformly with respect to the states of the two copies at time $n_0$}. We get in fact a stronger property than I-cosiness, which we call \emph{uniform cosiness}, and which implies not only the I-cosiness of the filtration, but also the ergodicity of the Markov kernel.
   
   Recall that a Markov kernel on a compact metric space $(\X,d)$ is \emph{Feller} if $x\mapsto P(x,\cdot)$ is continuous for the weak* topology. Any Feller Markov kernel on a compact metric space admits an invariant probability distribution on $X$. The Markov kernel is said to be \emph{ergodic in the Markov sense} if there exists a probability distribution $\mu$ on $X$ such that, for each probability measure $\nu$ on $\X$, 
   \[ \nu Q^n  \tend[w*]{n}{\infty} \mu.\]
   In this case, $\mu$ is the unique ${Q}$-invariant probability measure.

    \begin{definition}[Uniform cosiness]
     Let  ${Q}$ be a Markov kernel on a compact metric space $({\X},d)$. We say that ${Q}$ is \emph{uniformly cosy} if, for each $\delta>0$, there exist $M=M(\delta)>0$ such that, whenever $n_0$ is a negative integer with $|n_0|\ge M$, for each $x',x''\in {\X}$ , there exists a probability measure $m_{x',x''}$ on $\bigl({\X}^{\{n_0,\ldots,0\}}\bigr)^2$ , depending measurably on $(x',x'')$ such that, denoting by $(X'_n)_{n_0\leq n\leq0}$ and $(X''_n)_{n_0\leq n\leq0}$ the canonical processes defined by the coordinates on $\bigl({\X}^{\{n_0,\ldots,0\}}\bigr)^2$, the following conditions hold:
     \begin{itemize}
      \item The starting points of the processes are given by $X'_{n_0}=x'$ and $X''_{n_0}=x''$, $m_{x',x''}$-almost surely.
      \item Under $m_{x',x''}$, for each $n_0\leq n\leq -1$, the conditionnal distribution 
      $$ \L\bigl( (X'_{n+1},X''_{n+1}) \,|\, \tribu(X'_{n_0},\ldots,X'_n,X''_{n_0},\ldots,X''_n)\bigl) $$
      is almost surely a coupling of the two probability distributions ${Q}(X'_n,\cdot)$ and ${Q}(X''_n,\cdot)$ (we thus realize from $n_0$ a real-time coupling of two Markov processes of transitions probabilities given by ${Q}$; one starting at $x'$ and the other at $x''$).
      \item We have
      \[ m_{x',x''}\big(d(X'_0,X''_0)>\delta\bigr)<\delta. \]
     \end{itemize}
    \end{definition}
    
    For example, the proof of Theorem~\ref{thm:fte_standard} shows that the Markov kernel defined by the probabilistic cellular automaton $\te$ is uniformly cosy.

%
%
%

\begin{theo}
	 Let ${Q}$ be a Feller Markov kernel on the compact metric space $({\X},d)$. If ${Q}$ is uniformly cosy, then ${Q}$ is ergodic in the Markov sense.  	
\end{theo}

\begin{proof}
Let $\mu$ be a $Q$-invariant probability measure on ${\X}$, and let $\nu$ be any probability measure on  ${\X}$. We want to prove that, for each continuous function $f:{\X}\to\RR$ and each $\eps>0$, for each $n$ large enough we have
\[ \left|\int_{{\X}} f\, d(\nu Q^n) - \int_{{\X}} f\, d\mu\right| \le\ \eps. \] 

Given $f$ and $\eps$, by uniform continuity of $f$ on the compact $\X$, we get $\delta>0$  such that :
\[
 \forall x,y\in{\X},\ d(x,y)\leq \delta\Longrightarrow \left| f(x)-f(y) \right| < \eps/2.
\]
We can also assume that $\delta\|f\|_\infty<\eps/4$. 

By uniform cosiness, we can take $M>0$ such that, given any integer $n_0\leq -M$, there exists a family of probability measures $\left(m_{x',x''}\right)_{x',x''\in{\X}}$ associated to this $\delta$ as in the statement of the definition of uniform cosiness. We then define the probability measure $m_{\mu,\nu}$ on $\bigl({\X}^{\{n_0,\ldots,0\}}\bigr)^2$ by
\[ m_{\mu,\nu} := \int_{{\X}} \int_{{\X}} m_{x',x''}\, d\mu(x')\, d\nu(x''). \]
Under $m_{\mu,\nu}$, the processes $(X'_n)_{n_0\leq n\leq0}$ and $(X''_n)_{n_0\leq n\leq0}$ are two Markov chain with transitions given by $Q$, with respective initial laws $\L(X'_{n_0})=\mu$ and $\L(X''_{n_0})=\nu$. We then have
$\L (X'_{0})=\mu{Q}^{|n_0|}\mu = \mu$ by invariance of $\mu$, and  $ \L (X''_{0})={\nu Q}^{|n_0|}$. 
The uniform cosiness yields
\[ m_{\mu,\nu}\big(d(X'_0,X''_0)>\delta\big) 
= \int_{{\X}} \int_{{\X}} \underbrace{m_{x',x''}\big(d(X'_0,X''_0)>\delta\big)}_{<\delta}\, d\mu(x')\, d\nu(x'') < \delta. \]
We then have 
\begin{align*}
 \left|\int_{{\X}} f\, d(\nu{Q}^{|n_0|}) - \int_{{\X}} f\, d\mu\right|
& \leq \EE_{m_{\mu,\nu}}\left[ \bigl| f(X''_0)-f(X'_0) \bigr| \right] \\
& \leq \int_{d(X'_0,X''_0)>\delta} \bigl| f(X''_0)-f(X'_0) \bigr| \,dm_{\mu,\nu} + \eps/2 \\
& \le 2\delta\|f\|_\infty + \eps/2 \le \eps.
\end{align*}
\end{proof}

\begin{corollary}
\label{cor:te_ergodic}
 The Markov kernel given by the cellular automaton $\te$ is ergodic in the Markov sense, and $\mu=\U_\AA^{\otimes\ZZ}$ is the unique $\te$-invariant probability measure.
\end{corollary}

\section{Dynamical classification of factor filtrations}

\subsection{Factor filtrations} 

In the two cases we have studied above, we observe that the filtrations $\ftm$ and $\fte$ both enjoy an important property: if we consider their canonical construction, which is on $\AA^\ZZ$ for $\ftm$ and on $(\AA^\ZZ)^\ZZ$ for $\te$, we have on the ambient probability space the action of the left shift $\sigma$, which is a measure-preserving transformation, and moreover all the sigma-algebras in those filtrations are invariant with respect to this automorphism. Our purpose in this section is to formalize such a situation, and adapt the study of the filtration to this context. 

We call here \emph{dynamical system} any system of the form $(\Omega,\PP,T)$, where $(\Omega,\PP)$ is a probability space, and $T:\Omega\to\Omega$ is an invertible, bi-measurable transformation which preserves the probability measure $\PP$. Given such a system, we call \emph{factor sigma-algebra} of $(\Omega,\PP,T)$ any sub-sigma algebra $\A$ of the Borel sigma algebra of $\Omega$ which is invariant by $T$: for each $A\in\A$, $T^{-1}A$ and $TA$ are also in $\A$.
For any random variable $X_0$ defined on $(\Omega,\PP)$, as $T$ preserves $\PP$, the process $X=\bigl(X(i)\bigr)_{i\in\ZZ}$ defined by 
\begin{equation}
 \label{eq:T-process}
 \forall i\in\ZZ,\ X(i):=X_0\circ T^i
\end{equation}
is stationary. Such a stationary process will be called \emph{a $T$-process}. Whenever $X$ is a $T$-process, the sigma-algebra generated by $X$ is clearly a factor sigma-algebra of $(\Omega,\PP,T)$. Conversely, any factor sigma-algebra of $(\Omega,\PP,T)$ is generated by some $T$-process (remember that all sigma-algebras are assumed to be essentially separable). 

\begin{definition}
 \label{def:factor_filtration}
 We call \emph{factor filtration} any pair $(\F,T)$ where
 \begin{itemize}
  \item $\F=\bigl(\F_n\bigr)_{n\le0}$ is a filtration on some probability space $(\Omega,\PP)$,
  \item $(\Omega,\PP,T)$ is a dynamical system,
  \item for each $n\le0$,  $\F_n$ is a factor sigma-algebra of $(\Omega,\PP,T)$. 
 \end{itemize}
\end{definition}

In view of the above discussion, in a factor filtration $(\F,T)$ the filtration $\F$ is always generated by a process $(X_n)_{n\le0}$, where for each $n$, $X_n=\bigl(X_n(i)\bigr)_{i\in\ZZ}$ is a $T$-process.

The dynamical classification of factor filtrations that we want to introduce now aims at distinguishing these objects up to the following notion of isomorphism, which is the adaptation of Definition~\ref{def:isomorphism} to the dynamical context.

\begin{definition}[Dynamical isomorphism of factor filtrations]
 \label{def:dyn_isomorphism}
 Let $(\F,T)$ and $(\F',T')$  be two factor filtrations, possibly defined on two different dynamical systems. We say that they are \emph{dynamically isomorphic} if we can find two processes $X=(X_n)_{n\leq0}$ and $X'=(X'_n)_{n\leq0}$  generating respectively $\F$ and $\F'$, and such that
 \begin{itemize}
  \item for each $n$, $X_n=\bigl(X_n(i)\bigr)_{i\in\ZZ}$ is a $T$-process, and $X'_n=\bigl(X'_n(i)\bigr)_{i\in\ZZ}$ is a $T'$-process;
  \item $\L(X)=\L(X')$.
 \end{itemize}
 We write in this case $(\F,T)\sim(\F',T')$, or $(\F,T)\isoproc{X}{X'}(\F',T')$ to specify the processes involved in the isomorphism. 
\end{definition}

We note the following specificity of dynamical isomorphism: if  $(\F,T)\isoproc{X}{X'}(\F',T')$, and if $Y'$ is the copy of an $\F_0$ random variable $Y$ provided by this isomorphism, then $Y'\circ T'$ is the corresponding copy of $Y\circ T$. 

\medskip

The notion of immersion is unchanged: $(\F,T)$ and $(\G,T)$ being two factor filtrations in the same dynamical system, we simply say that $(\F,T)$ is immersed in $(\G,T)$ if $\F$ is immersed in $\G$. However, the notion of immersibility takes into account the above definition of dynamical isomorphism.

\begin{definition}[Dynamical immersibility of factor filtrations]
 \label{def:dyn_imm}
 Let $(\F,T)$ and $(\G,S)$  be two factor filtrations, possibly defined on two different dynamical systems. We say that $(\F,T)$ is dynamically immersible in $(\G,S)$ if there exists a factor filtration $(\F',S)$ immersed in $(\G,S)$, which is dynamically isomorphic to $(\F,T)$. 
\end{definition}

We can now adapt the notions of product type and standardness to factor filtrations.

\begin{definition}
 \label{def:pt}
 The factor filtration $(\F,T)$ is said to be \emph{dynamically of product type} if there exists a family $(\G_n)_{n\le0}$ of independent factor sigma-algebras such that for each $n\le0$, $\F_n=\bigvee_{m\le n}\G_m$.
\end{definition}

Equivalently, $(\F,T)$ is of product type if $\F$ can be generated by a process $X=(X_n)_{n\leq0}$ whose coordinates $X_n$ are independent $T$-processes. Of course, if the factor filtration $(\F,T)$ is of product type, then the filtration $\F$ is itself of product type, but the converse is not true (see the example in Section~\ref{sec:dynftm}).

\begin{definition}
 \label{def:dstd}
 The factor filtration $(\F,T)$ is said to be \emph{dynamically standard} if it is dynamically immersible in some  factor filtration dynamically of product type.
\end{definition}

It is natural also to translate the notions of real-time coupling and I-cosiness to the dynamical context.
The corresponding notions are formally the same as in the static case, except that we have to replace the isomorphism of filtrations by the dynamical isomorphism of factor filtrations:

\begin{definition}[Dynamical real-time coupling]
\label{def:drtc}
 Let $(\F,T)$ be a factor filtration on a dynamical system $(\Omega,\PP,T)$. We call \emph{dynamical real-time coupling of $(\F,T)$} a pair $\bigl((\F',S),(\F'',S)\bigr)$ of factor filtrations, both defined on the same dynamical system (but possibly different from $(\Omega,\PP,T)$), such that 
 \begin{itemize}
  \item $(\F',S)\sim(\F,T)$,
  \item $(\F'',S)\sim(\F,T)$,
  \item $\F'$ and $\F''$ are jointly immersed.
 \end{itemize}
\end{definition}

\begin{definition}[Dynamical I-cosiness]
 \label{def:dyn-I-cosiness}
 Let $(\F,T)$ be a factor filtration, and let $Y$ be an $\F_0$-measurable random variable taking values in some Polish metric space $(E,d)$. 
 Then $Y$ is said to be \emph{dynamically I-cosy with respect to $(\F,T)$} if, for each real number $\delta>0$, we can find a dynamical real-time coupling $\bigl((\F',S),(\F'',S)\bigr)$ of $(\F,T)$ which is independent in the distant past, and such that the two copies $Y'$ and $Y''$ of $Y$ in $\F'$ and $\F''$ respectively satisfy
 \begin{equation}
  \label{eq:close_copies_dyn}
  d(Y',Y'')<\delta \text{ with probability } > 1-\delta.
 \end{equation}
 The factor filtration $(\F,T)$ is said to be \emph{dynamically I-cosy} if each $\F_0$-measurable random variable $Y$ is dynamically I-cosy with respect to $(\F,T)$.
\end{definition}

We expect of course some relationship between dynamical I-cosiness and dynamical standardness. In the static case, showing that standardness implies I-cosiness is not very difficult, and in fact exactly the same arguments apply in the dynamical case. We provide them below.

\begin{lemma}
 \label{lemma:dpt_implies_dic}
 Let $(\F,T)$ be a factor filtration which is dynamically of product type. Then $(\F,T)$ is dynamically I-cosy.
\end{lemma}
\begin{proof}
 Let $(\Omega,\PP,T)$ be the dynamical system where the factor filtration is defined, and let $X=(X_n)_{n\le0}$ be a process generating $\F$, where the coordinates $X_n$ are independent $T$-processes: $X_n=\bigl(X_n(i)\bigr)_{i\in\ZZ}=\bigl(X_n(0)\circ T^i\bigr)_{i\in\ZZ}$. 
 
 We have to check the dynamical I-cosiness with respect to $(\F,T)$ of a target random variable $Y$ which is $\F_0$-measurable. It is enough to consider the case where $Y$ is measurable with respect to $\tribu(X_{n_0+1},\ldots,X_0)$ for some $n_0\le -1$. 
 
 On the product dynamical system $(\Omega\times\Omega,\PP\otimes\PP,T\times T)$, we have two independent copies $X^{(1)}$ and $X^{(2)}$ of $X$, whose coordinates are independent $(T\times T)$-processes. Let us now consider the two  copies $X'$ and $X''$ of $X$ defined by $X':=X^{(1)}$, and 
 \[ X''_n := \begin{cases}
              X^{(2)}_n & \text{ if }n\le n_0,\\
              X^{(1)}_n & \text{ if }n\ge n_0+1.\
             \end{cases}
\]
Define $\F'$ (respectively $\F''$) as the filtration generated by $X'$ (respectively $X''$). Then $\bigl((\F',T\times T),(\F'',T\times T)\bigr)$ is an $n_0$-independent dynamical coupling of $(\F,T)$. Moreover, since by construction we have $X'_n=X''_n$ for $n\ge n_0+1$, the corresponding copies of $Y'$ and $Y''$ of $Y$ satisfy $Y'=Y''$. This concludes the proof.
\end{proof}

\begin{lemma}
 \label{lemma:dim_dic}
 Let $(\F,T)$ and $(\G,S)$ be two factor filtrations. Assume that $(\G,S)$ is dynamically I-cosy, and that $(\F,T)$ is dynamically immersible in $(\G,S)$. Then $(\F,T)$ is dynamically I-cosy.
\end{lemma}

\begin{proof}
Since cosiness is preserved by isomorphism, we can assume without loss of generality that $(\F,T)$ is \emph{immersed} in some dynamically I-cosy factor filtration $(\G,T)$. Let $X=(X_n)_{n\le0}$ be a process generating $\F$, where the coordinates $X_n$ are $T$-processes.
 Let $Y$ be an $\F_0$-measurable random variable, taking values in the Polish metric space $(E,d)$, and let $\delta>0$. Since $\F_0\subset \G_0$, and since $(\G,T)$ is dynamically I-cosy, there exists a dynamical real-time coupling $\bigl((\G',S),(\G'',S)\bigr)$ of $(\G,T)$, independent in the distant past, such that the corresponding copies $Y'$ and $Y''$ of $Y$ satisfy $d(Y',Y'')<\delta$ with probability at least $1-\delta$. But the generating process $X$ is $\G_0$-measurable, so we can consider its copies $X'$ and $X''$ in $\G'$ and $\G''$ respectively.
 Let $\F'$ (respectively $\F''$) be the filtration generated by $X'$ (respectively $X''$). Since these copies are provided by a dynamical isomorphism, the coordinates of $X'$ and $X''$ are $S$-processes, hence we get two dynamically isomorphic copies $(\F',S)$ and $(\F'',S)$ of the factor filtration $(\F,T)$.  Since $\G'$ and $\G''$ are jointly immersed, and $\F'$ (respectively $\F''$) is immersed in $\G'$ (respectively $\G''$), $\F'$ and $\F''$ are jointly immersed. Hence we get a dynamical real-time coupling $\bigl((\F',S),(\F'',S)\bigr)$ of the factor filtration $(\F,T)$. Moreover this coupling is independent in the distant past because this is the case for $\bigl((\G',S),(\G'',S)\bigr)$. But $Y'$ and $Y''$ are respectively $\F'_0$ and $\F''_0$-measurable, as $Y$ is $\F_0$-measurable, so they are also the corresponding copies of $Y$ in the dynamical real-time coupling $\bigl((\F',S),(\F'',S)\bigr)$.
\end{proof}

From Lemma~\ref{lemma:dpt_implies_dic} and Lemma~\ref{lemma:dim_dic}, we immediately derive the following theorem:

\begin{theo}
 \label{thm:dstd_implies_dic}
 If the factor filtration $(\F,T)$ is dynamically standard, then it is dynamically I-cosy.
\end{theo}

Whether the converse of the above theorem is true is still for us an open question (see a partial result in this direction in Section~\ref{sec:converse}). However, the fact that dynamical I-cosiness is necessary for dynamical standardness can already be used to establish that some factor filtrations are not dynamically standard, as in the following section.

\subsection{Dynamical vs. static classification of filtrations: the example of $\ftm$}
\label{sec:dynftm}

We come back again to the filtration $\ftm$ associated to the deterministic cellular automaton $\tau$, and defined in Section~\ref{sec:ftm}. 
The probability measure $\mu=\U_\AA^{\otimes\ZZ}$ on $\AA^\ZZ$ is invariant by the shift map $\sigma$. Furthermore, since the transformations $\tau$ and $\sigma$ commute, each coordinate $X_n=\bigl(X_n(i)\bigr)_{i\in\ZZ}$  of the process generating the filtration $\ftm$ is itself a stationary $\sigma$-process: for each $i\in\ZZ$, $X_n(i)=X_n(0)\circ \sigma^i$. Thus, $\bigl(\ftm,\sigma\bigr)$ is a factor filtration in the dynamical system $\bigl(\AA^\ZZ,\mu,\sigma\bigr)$.

We recall that, if we look at the filtration $\ftm$ from the static point of view, it is of product type (Theorem~\ref{thm:ftm_product_type}). However the following result shows that dynamical classification of factor filtrations may lead to different results than the static one.

\begin{theo}
\label{thm:ftm_not_dyn_standard}
 The factor filtration $(\ftm,\sigma)$ is not dynamically standard.
\end{theo}

\begin{proof}
 The strategy consists in showing that the factor filtration $(\ftm,\sigma)$, and more precisely the target $X_0(0)$, is not dynamically I-cosy. For this, we consider a real-time dynamical coupling of
 this factor filtration in a dynamical system $(\Omega,\PP,T)$, so that we have two copies $X'=(X'_n)_{n\leq0}$ and $X''=(X''_n)_{n\leq0}$ of the process $(X_n)$, where for each $n\leq0$, 
 \begin{itemize}
  \item $X'_n=\bigl(X'_n(i)\bigr)_{i\in\ZZ}=\tau^{|n|} X'_0$,
  \item $X''_n=\bigl(X''_n(i)\bigr)_{i\in\ZZ}=\tau^{|n|} X''_0$,
 \end{itemize}
   and for each $i\in\ZZ$, 
\begin{itemize}
 \item $X'_n(i)=X'_n(0)\circ T^i$,
 \item $X''_n(i)=X''_n(0)\circ T^i$.
\end{itemize}
For each $n\leq0$, set $Z_n:=X'_n-X''_n$. Then for each $n$, $Z_n=\bigl(Z_n(i)\bigr)_{i\in\ZZ}=\bigl(Z_n(0)\circ T^i\bigr)_{i\in\ZZ}$ is a stationary $T$-process taking values in $\AA$. And since $\tau$ is an endomorphism of the group $\AA^\ZZ$, we also have $Z_n=\tau^{|n|}Z_0$ for each $n\leq0$. 

Assume that this coupling is $n_0$-independent for some $n_0\leq0$, in other words $X'_{n_0}$ and $X''_{n_0}$ are independent. Then the probability distribution of $Z_{n_0}$ is nothing but $\mu=\U_\AA^{\otimes\ZZ}$. Therefore, the Kolmogorov-Sinaï entropy of this $T$-process is 
\[ h(Z_{n_0},T)=\log\cA.\] 
But we have 
\[ Z_{n_0}=\tau^{|n_0|}Z_0, \]
so that the $T$-process $Z_{n_0}$ is a factor of the $T$-process $Z_{0}$. In particular, since the Kolmogorov-Sinaï entropy can not increase when we pass to a factor, and observing that $Z_0$ is a $T$ process taking values in $\AA$  
\[ \log\cA = h(Z_{n_0},T)\leq h(Z_{0},T)\leq \log\cA.\]
However the only $T$-process taking values in $\AA$ and whose entropy is $\log\cA$ is the uniform Bernoulli process. It follows that the probability distribution of $Z_0$ is necessarily also $\mu$. Therefore, 
\[ \PP\bigl(X'_0(0)=X''_0(0)\bigr) = \PP\bigl(Z_0(0)=0\bigr) = \frac{1}{\cA},  \]
which prevents the dynamical I-cosiness of the target $X_0(0)$.
\end{proof}

\subsubsection{An alternative proof for the case of an infinite group}
In this section only we allow the non-trivial Abelian group $(\AA,+)$ to be metrizable compact, not necessarily finite. The action of the deterministic cellular automaton $\tau$ on $\AA^\ZZ$ can be defined by the same formula as in the finite case~\eqref{def:tau}. We denote by $\U_\AA$ the Haar measure on $\AA$ (which is the uniform measure in the case of a finite group). Still, $\mu=\U_\AA^{\otimes\ZZ}$ is invariant by $\tau$ and by the shift map $\sigma$, and we can consider the filtration $\ftm$ defined in Section~\ref{sec:ftm}, as well as the factor filtration $(\ftm,\sigma)$ in this more general context. Like in the finite case, we can show by exactly the same arguments that $\ftm$ is of product type even in the non-finite case. But the above proof of Theorem~\ref{thm:ftm_not_dyn_standard} makes an essential use of the fact that $\AA$ is a \emph{finite} Abelian group. We propose below an alternative proof, which replaces the entropy argument by a coupling argument, and which is still valid in the infinite case. 

We first need to recall some basic facts about joinings of stationary processes in ergodic theory.

\begin{definition}[Joining]
Let $I$ be a finite or countable set of indices. For $i\in I$, let $X_i$ be a $T_i$-process in a dynamical system $(\Omega_i,\PP_i,T_i)$. We call \emph{dynamical coupling} or \emph{joining} of the processes $(X_i)_{i\in I}$ the simultaneous realization \emph{in the same dynamical system} $(\overline\Omega,\overline\PP,\overline T)$ of $\overline T$-processes $\overline X_i$, $i\in I$, such that $\L(\overline X_i)=\L(X_i)$ for each  $i\in I$.
\end{definition}

\begin{definition}[Disjointness]
 The stationary processes $X_1$ and $X_2$ are said to be \emph{disjoint} if, for each joining $(\overline X_1,\overline X_2)$ of $X_1$ and $X_2$, $\overline X_1$ and $\overline X_2$ are independent.
\end{definition}

Disjointness was introduced by Furstenberg in 1967 in his famous paper~\cite{furstenberg67}, where he also provided one of the first examples of disjointness:

\begin{theo}
 \label{thm:disjonction}
 Let $X_1$ be a Bernoulli process (\textit{i.e.} a stationary process with independent coordinates), and let $X_2$ be a stationary process with zero entropy. Then $X_1$ and $X_2$ are disjoint.
\end{theo}

The following result is a direct application of the construction of a relatively independent joining over a common factor (see for example~\cite{joinings}).

\begin{lemma}
 \label{lemma:cri}
 Let $X_1$ and $Y_1$ be two $T_1$-processes in a dynamical system $(\Omega_1,\PP_1,T_1)$, and let Let $X_2$ and $Y_2$ be two $T_2$-processes in a second dynamical system $(\Omega_2,\PP_2,T_2)$. Assume that $\L(Y_1)=\L(Y_2)$. then there exists a joining  $(\overline{X}_1,,\overline{Y}_1,\overline{X}_2,\overline{Y}_2)$ of $X_1$, $Y_1$, $X_2$, $Y_2$, such that
 \begin{itemize}
  \item $\L(\overline{X}_1,\overline{Y}_1)=\L({X}_1,{Y}_1)$,
  \item $\L(\overline{X}_2,\overline{Y}_2)=\L({X}_2,{Y}_2)$,
  \item $\overline{Y}_1=\overline{Y}_2$ a.s.  
 \end{itemize}
\end{lemma}

The invariance of $\mu$ by $\tau$ tells us that, if a random variable $X$ takes its values in $\AA^\ZZ$ and follows the distribution $\mu$, then the law of $\tau X$ is also $\mu$. The following lemma, which is the key ingredient in the alternative proof, can be viewed as the reciprocal of this fact in the stationary case.

\begin{lemma}
 \label{lemma:reciprocal}
 Let $X$ be a random variable taking values in $\AA^\ZZ$, whose law is $\sigma$-invariant. If $\L(\tau X)=\mu$, then $\L(X)=\mu$.
\end{lemma}

\begin{proof}
 Let $\rho$ be a $\sigma$-invariant probability measure on $\AA^\ZZ$, whose pushforward measure by $\tau$ is  $\tau_*(\rho)=\mu$. In the dynamical system $\bigl(\AA^\ZZ,\rho,\sigma\bigr)$, we consider the $\sigma$-process $X_1$ defined by the coordinates, and we set $Y_1:=\tau(X_1)$, so that $Y_1$ is a $\sigma$-process with law $\mu$. 
 We also consider a second dynamical system $\bigl(\AA^\ZZ,\mu,\sigma\bigr)$, where we define the $\sigma$-process $X_2$ by the coordinates, and we set $Y_2:=\tau(X_2)$. We have $\L(X_2)=\L(Y_2)=\mu$. By Lemma~\ref{lemma:cri} there exists a joining $(\overline{X}_1,,\overline{Y}_1,\overline{X}_2,\overline{Y}_2)$ realized in a dynamical system $(\overline\Omega,\overline\PP,\overline T)$, such that:
\begin{itemize}
 \item $\L(\overline{X}_1)=\rho$,
 \item $\L(\overline{X}_2)=\mu$,
 \item $\tau(\overline{X}_1)=\overline{Y}_1=\overline{Y}_2=\tau(\overline{X}_2)$.
\end{itemize}
Set $Z:=\overline{X}_1-\overline{X}_2$. We have $\tau(Z)=\tau(\overline{X}_1)-\tau(\overline{X}_2)=0$, therefore for each $i\in\ZZ$, 
\[ Z(i)=Z(0)\circ \overline{T}^i= (-1)^i Z(0). \]
The process $Z$ being 2-periodic, its Kolmogorov-Sinaï entropy vanishes: 
\[ h(Z,\overline{T})=0. \] 
Moreover, $\overline{X}_2$ has independent coordinates since its law is $\mu$. By Theorem~\ref{thm:disjonction}, $\overline{X}_2$ is disjoint of $Z$, hence independent of $Z$. But we have $\overline{X}_1=Z+\overline{X}_2$, and the law of $\overline{X}_2$ is the Haar measure $\mu$ on $\AA^\ZZ$. It follows that the law $\rho$ of $\overline{X}_1$ is also $\mu$.
\end{proof}

Note that the assumption of $\sigma$-invariance in Lemma~\ref{lemma:reciprocal} cannot be removed, because we could take a random variable $Y$ of law $\mu$, and define $X$ as the only preimage of $Y$ by $\tau$ whose coordinate $X_0$ satisfies $X_0=0_\AA$. Then we would have $\L(\tau X)=\L(Y)=\mu$, but $\L(X)\neq \mu$ since $\AA$ is not the trivial group.

We now have all the tools to modify the the end of the proof of Theorem~\ref{thm:ftm_not_dyn_standard}, so that it can still be valid in the infinite case.

\begin{proof}[Alternative proof of Theorem~\ref{thm:ftm_not_dyn_standard}]
 We consider the same situation as in the first proof, but we just modify the final argument. 
 We know that $Z_{n_0}$ has law $\mu$, and $\tau^{|n_0|}Z_0=Z_{n_0}$. All processes are stationary, so we can apply Lemma~\ref{lemma:reciprocal} $|n_0|$ times to prove recursively that each $Z_n$, $n_0\le n\le 0$, has law $\mu$.  Observe also that we can choose the metric $d$ on $\AA$ to be invariant by translation (if we start with an arbitrary metric $d'$ defining the topology on $\AA$, we can set $d(a,b):=d'(a-b,b-a)$). Then, for each $\delta>0$, denoting by $B_\delta$  the ball of center $0_\AA$ and with radius $\delta$, we get 
\[ \PP\bigl(d\bigl(X'_0(0),X''_0(0)\bigr)<\delta\bigr)= \PP\bigl(Z_0(0)\in B_\delta\bigr)=\U_\AA\bigl(B_\delta\bigr). \]
But, as $\delta\to0$, $\U_\AA\bigl(B_\delta\bigr)$ converges to $\U_\AA(\{0_\AA\})$, and we have $\U_\AA(\{0_\AA\})<1$ since $\AA$ is not the trivial group. This shows that the target $X_0(0)$ is not dynamically I-cosy.
\end{proof}

\subsection{The factor filtration $(\fte,\sigma)$}
\label{sec:fte_dyn_standard}
We come back now to the case of a finite Abelian group, and we wish to study the dynamical standardness of the factor filtration $(\fte,\sigma)$, where $\fte$ is the filtration associated to the probabilistic cellular automaton $\te$. This section is essentially devoted to the proof of the following result.

\begin{theo}
 \label{thm:fte_dyn_standard}
 There exists $0<\eps_0<\frac{\cA-1}{\cA}$ such that, for each $\eps_0<\eps\le\frac{\cA-1}{\cA}$,  $(\fte,\sigma)$ is dynamically standard. 
\end{theo}

For this we will construct the filtration $\fte$ on a bigger space than $(\AA^\ZZ)^\ZZ$, namely we will consider the probability space $\Omega:=[0,1]^{\ZZ\times\ZZ}$ equipped with the product probability measure $\PP:=\lambda^{\otimes (\ZZ\times\ZZ)}$ ($\lambda$ denotes the Lebesgue measure on $[0,1]$ here). On $(\Omega,\PP)$ we define the measure preserving $\ZZ^2$-action of shift maps:
for $m,j\in\ZZ$, we define 
\[ \sigma_{m,j}:\omega=(\omega_{n,i})_{n\in\ZZ,i\in\ZZ}\mapsto \sigma_{m,j}\omega:=(\omega_{n+m,i+j})_{n\in\ZZ,i\in\ZZ}, \]
and we simply denote by $\sigma$ the usual left-shift: $\sigma:=\sigma_{0,1}$.
For each $n\in\ZZ$, $i\in\ZZ$, we set $U_n(i)(\omega):=\omega_{n,i}$, so that $\bigl(U_n(i)\bigr)_{n\in\ZZ,i\in\ZZ}$ is a family of independent variables, uniformly distributed in $[0,1]$. For each $n\in\ZZ$, we denote by $U_n$ the $\sigma$-process $U_n:=\bigl(U_n(i)\bigr)_{i\in\ZZ}$. Let $\G=(\G_n)_{n\leq0}$ be the filtration generated by $(U_n)_{n\leq0}$: then $(\G,\sigma)$ is a factor filtration, and as the random variables $U_n$, $n\in\ZZ$, are independent,  $(\G,\sigma)$ is dynamically of product type. 

Our purpose is to construct, in the dynamical system $(\Omega,\PP,\sigma)$, a stationary Markov process $(X_n)_{n\in\ZZ}$ with transitions given by the Markov kernel $\te$. We will still denote by $\fte$ the natural filtration of the negative-time part of this process, and we want to make it in such a way that $(\fte,\sigma)$ be a factor filtration immersed in $(\G,\sigma)$. More precisely, we want the following five conditions to be realized:
\begin{itemize}
\label{p:conditions}
 \item[(a)] $X_n=\bigl(X_n(i)\bigr)_{i\in\ZZ}=\bigl(X_n(0)\circ\sigma^i\bigr)_{i\in\ZZ}$ is a $\sigma$-process,
 \item[(b)] $\L(X_n)=\mu$,
 \item[(c)] for each $x\in\AA^\ZZ$, $\L(X_{n+1}|X_n=x)=\te(x,\cdot)$,
 \item[(d)] $X_n$ is $\G_n$-measurable, namely, is a measurable function of $(U_m)_{m\le n}$,
 \item[(e)] $X_{n+1}$ is measurable with respect to $\tribu(X_n,U_{n+1})$.
\end{itemize}
Conditions (a), (b) and (c) ensure that $(X_n)_{n\le0}$ generates the correct factor filtration $(\fte,\sigma)$ that we want to study.
Condition (d) shows that $\fte$ is a sub-filtration of $\G$, and Condition (e) implies that $\fte$ is immersed in $\G$.

The key condition~(d) is the most difficult to obtain. Our strategy will be strongly inspired by clever techniques introduced by Marcovici~\cite{MarcoArticle}  to prove ergodicity in the Markov sense for some probabilistic cellular automata. Let us start by sketching the argument. It relies on the following fact, which is a straightforward consequence of~\eqref{eq:law_of_xi}:
for each $n,i\in\ZZ$, the conditional distribution of $X_{n+1}(i)$ always satisfies
\begin{equation}
\label{eq:unif_minoration}
 \L\bigl(X_{n+1}(i) \,|\, \fte[n]\bigr) = \L\bigl(X_{n+1}(i) \,|\, X_n(i),X_n(i+1)\bigr) \ge \epst\, \U_\AA.
\end{equation}
Therefore, we are able to parametrize our Markov chain with the uniform random variables $U_n(i)$, through a recursion relation of the form
\[
 X_{n+1}(i)=\phiep\bigl(X_n(i),X_n(i+1),U_{n+1}(i)\bigr),
\]
where the updating function $\phiep(\cdot,\cdot,u)$ only depends on $u$ provided $u\ge 1-\epst$: namely we split the interval $[1-\epst,1]$ into $\cA$ subintervals $I_a$ with length $\epst/\cA$ each, and set for each $a,b,c\in\AA$ and $u\in I_a$
\[
 \phiep(b,c,u):=a.
\]
(The complete construction of the updating function is given in Section~\ref{sec:updating}.)
In this way, the knowledge of the process $U$ is sufficient to recover the values of the random variables $X_n(i)$ at some good sites $(n,i)$, directly if $U_n(i)>\epst$, or indirectly through the recursion relation. Actually, these good sites are exactly those which are not connected to $-\infty$ in an oriented site percolation process introduced in the next section. We will see that, if $\epst$ is not too small, then with probability one there is no site connected to $-\infty$, hence all sites are good and we get Condition~(d). 

The whole construction of the process also borrows from~\cite{MarcoArticle} the concept of \emph{envelope PCA}. It is an auxiliary probabilistic cellular automaton, presented in Section~\ref{sec:enveloppePCA}, which will be coupled with the percolation process in Section~\ref{sec:coupling_with_perco}.

\subsubsection{Oriented site percolation} 
\label{sec:perco}
The percolation model that we will use is built on an infinite oriented graph whose vertices (sites) are the elements of $\ZZ^2$, and the oriented edges connect each site $(n,i)$ to the sites $(n+1,i)$ and $(n+1,i-1)$. We fix $p\in[0,1]$ (this parameter will be specified later), and we use the random variables $U_n(i)$ to define a random configuration of the graph: the site $(n,i)$ is  declared to be \emph{open} if $U_n(i)\in[0,p[$, and \emph{closed} otherwise. In this random configuration, we say that a site $(n,i)$ \emph{leads to} a site $(m,j)$ if $n\le m$ and there exists in the oriented graph a path starting from $(n,i)$, ending in $(m,j)$ and passing only through open sites, which in particular requires that both sites $(n,i)$ and $(m,j)$ be open (see Figure~\ref{fig:perco}).

\begin{figure}
 \includegraphics[width=120mm]{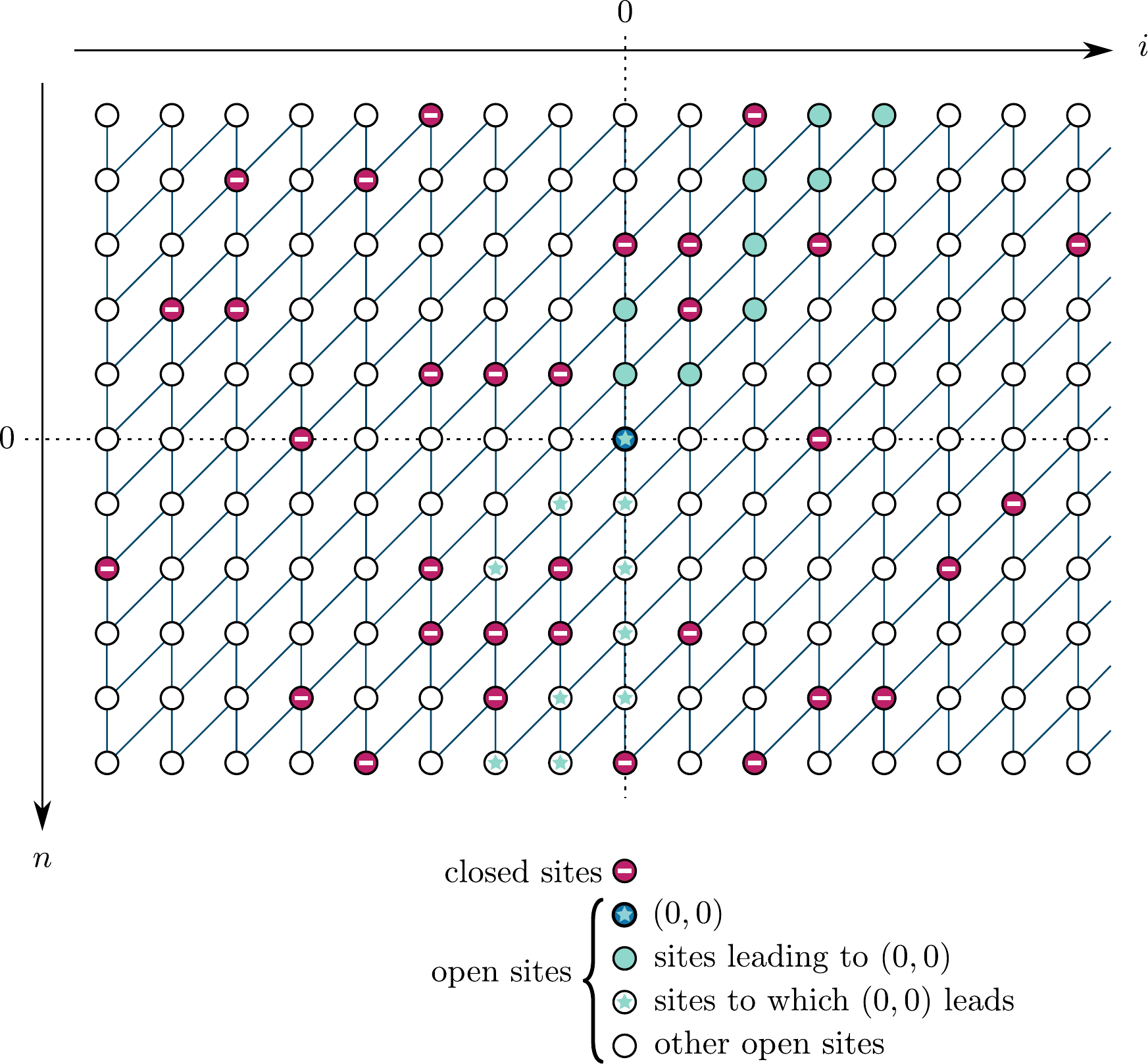}
\caption{The oriented site percolation model}
\label{fig:perco}
\end{figure} 

For $n\geq0$, we denote by $O_{n}$ the event ``there exists $i\in\ZZ$ such that $(0,0)$ leads to $(n,i)$'', and similarly, when $n\leq0$, $O_n$ stands for the events   ``there exists $i\in\ZZ$ such that $(n,i)$ leads to $(0,0)$''.  We also set 
\[
 O_\infty:= \bigcap_{n\geq0} O_n,\quad \text{and}\quad O_{-\infty}:= \bigcap_{n\leq0}O_n.
\]
We note that $\PP(O_n)=\PP(O_{-n})$ since exchanging each random variable $U_{n}(i)$ with $U_{-n}(-i)$ does not affect the law of the process, but exchanges $O_n$ and $O_{-n}$. It follows that $\PP(O_\infty)=\PP(O_{-\infty})$.

The essential question in percolation theory asks whether the probability of $O_\infty$ is positive or not. In the model presented here, this question has been addressed by Liggett~\cite{Liggett}, who proved the following:

\begin{theo}[Liggett]
\label{thm:perco_critique}
	There exists a critical value $p_c \in[\frac{2}{3},\frac{3}{4}]$ such that 
	\begin{itemize}
	 \item $p<p_c\Longrightarrow\PP(O_\infty)=0$,
	 \item $p>p_c\Longrightarrow\PP(O_\infty)>0$.
	\end{itemize}
\end{theo}

 \begin{corollary}
  \label{cor:perco}
  If $0\le p<p_c$, the integer-valued random variable
  \[
    N_0:=\min\{n\le0:O_n\text{ does not hold}\}
  \]
  is almost surely well defined, and it is measurable with respect to $\tribu\bigl(U_n,n\le0)$.
 \end{corollary}

\subsubsection{Envelope automaton} 
\label{sec:enveloppePCA}

The envelope automaton associated to $\te$ is an auxiliary probabilistic cellular automaton acting on configurations built on a larger alphabet: 
$\AAe:=\AA\cup\{\qm \}$. We associate to each finite word $w\in(\AAe)^*$ the set $\overline{w}$ of all words in $\AA^*$ obtained from $w$ by replacing each question mark with any symbol from $\AA$.
For example, if $\AA=\{0,1\}$ and $w=\qm 1\qm $, we have $\overline{w}=\{010,011,110,111\}$. Of course, if $w\in\AA^*$, we have $\overline{w}=\{w\}$.

To define the envelope automaton we first have to introduce the \emph{local transition rules} of $\te$: if $(X_n)$ is a Markov chain evolving according the Markov kernel $\te$, we define for $a,b,c\in\AA$
\begin{align*}
f_\te (a \,|\, bc) &:= \PP \bigl(X_{n+1}(i)=a \,|\, X_n(i)=b,X_n(i+1)=c\bigr) \\
 & = \begin{cases}
      1-\eps &\text{ if }a=b+c\\
      \frac{\eps}{\cA-1} &\text{ otherwise. }
     \end{cases}
\end{align*}
Now we define the local transition rules of the envelope PCA: if $w$ is a word of length 2 over the alphabet $\AAe$, and $a\in\AA$, we set
\[ \fenv(a \,|\, w) := \min_{bc\in\overline{w}} f_\te (a \,|\, bc), \]
and 
\[ \fenv(\qm  \,|\, w) := 1-\sum_{a\in\AA}\min_{bc\in\overline{w}} f_\te (a \,|\, bc). \]
Note that, if $w$ contains at least one $\qm $, for each $a\in\AA$ there exists $bc\in\overline{w}$ such that $b+c\neq a$. Recalling the assumption~\eqref{eq:epsilon} we made on $\eps$ and the definition~\eqref{eq:epst} of $\epst$, we thus get the following rules: for each $a,b,c\in\AAe$,
\[ 
 \fenv(a \,|\, bc) = \begin{cases}
                      f_\te(a\,|\,bc)    &\text{ if }a,b,c\in\AA,\\
                      0                  &\text{ if }a=\qm \text{ and }b,c\in\AA,\\
                      \frac{\eps}{\cA-1} &\text{ if }a\in\AA\text{ and }\qm \in\{b,c\},\\
                      1-\epst              &\text{ if }a=\qm \text{ and }\qm \in\{b,c\}.\\
                     \end{cases}
\]
In particular, the symbol $\qm $ cannot appear if the two preceding symbols are in $\AA$, and in this case the local transition rules of the envelope PCA coincides with those of $\te$. 

For each $b,c\in\AAe$, $\fenv(\cdot \,|\, bc)$ can be viewed as a probability measure on $\AAe$. Then the envelope PCA $\teenv$ is defined as the Markov kernel on $\AAe^\ZZ$ given by
\[ 
 \forall x\in\AAe^\ZZ,\ \teenv(\cdot\,|\,x):=\bigotimes_{i\in\ZZ}\fenv\bigl(\cdot \,|\, x(i) x(i+1)\bigr).
\]

\subsubsection{Simultaneous implementation of the two PCA's}
\label{sec:updating}
We are going to realize a coupling of $\te$ and $\teenv$ on our space $\Omega$, using the independent uniform random variables $\bigl(U_n(i)\bigr)_{n,i\in \ZZ}$. This will be done by means of two updating functions 
\[ \phiep:\AA\times\AA\times [0,1]\to\AA\]
and 
\[ \phiepenv:\AAe\times\AAe\times [0,1]\to\AAe,\]
adapted respectively to the local transition rules of $\te$ and $\teenv$. Here is how we define these updating functions (see Figure~\ref{fig:updating}): we split the interval $[0,1]$  into  $\cA+1$ disjoint sub-intervals, labelled by symbols from $\AAe$. The first one is
\[ I_\qm  := \left[0,1-\frac{\eps\cA}{\cA-1} \right[ = \left[0,1-\epst \right[, \]
and the complement of $I_\qm $ is itself split into $\cA$ subintervals of equal length $\frac{\eps}{\cA-1}$, denoted by  $I_a$, $a\in\AA$.
Now for each $b,c\in\AA$, $u\in[0,1]$, we set
\[ \phiep(b,c,u):= \begin{cases}
                                b+c &\text{ if } u \in I_\qm , \\
                                a &\text{ if } u \in I_a,\ a\in\AA.
                               \end{cases}
\]
In this way, for a random variable $U$ uniformly distributed in $[0,1]$, we have for each $b,c\in\AA$
\[ \L\bigl( \phiep(b,c,U) \bigr) = f_\te(\cdot\,|\,bc). \] 

%
%

To define $\phiepenv$, we use the same partition of $[0,1]$ and we set for $b,c\in\AAe$ and $u\in[0,1]$
\[ \phiepenv(b,c,u):= \begin{cases}
                                b+c &\text{ if } b,c\in\AA \text{ and } u \in I_\qm , \\
                                \qm  &\text{ if } b \text{ or }c\text{ is }\qm  \text{ and }u \in I_\qm , \\
                                a &\text{ if } u \in I_a,\ a\in\AA.
                               \end{cases}
\]
Likewise, we see that  for a random variable $U$ uniformly distributed in $[0,1]$, we have for each $b,c\in\AAe$
\[ \L\bigl( \phiepenv(b,c,U) \bigr) = \fenv(\cdot\,|\,bc). \] 

Let $n_0\le0$ be a fixed integer and $x\in\AA^\ZZ$ be an arbitrary configuration . We realize simultaneously two Markov chains $(X_n^{x})_{n\ge n_0}$ and $\Xenv=(\Xenv_n)_{n\ge n_0}$ evolving respectively according to $\te$ and $\teenv$: first we initialize by setting 
$X^{x}_{n_0} := x$, and $\Xenv_{n_0}:= \ldots \qm \qm \qm \qm \qm \ldots$ (the configuration with only symbols $\qm $). 
Later times of the processes are computed recursively through the formula
\begin{equation}
 \label{eq:rec_te}
 X^{x}_{n+1}(i):= \phiep\bigl(X^{x}_n(i),X^{x}_n(i+1),U_{n+1}(i)\bigr),
\end{equation}
and
\begin{equation}
 \label{eq:rec_env} 
 \Xenv_{n+1}(i):= \phiepenv\bigl(\Xenv_n(i),\Xenv_n(i+1),U_{n+1}(i)\bigr).
\end{equation}

\begin{figure}
 \includegraphics[width=125mm]{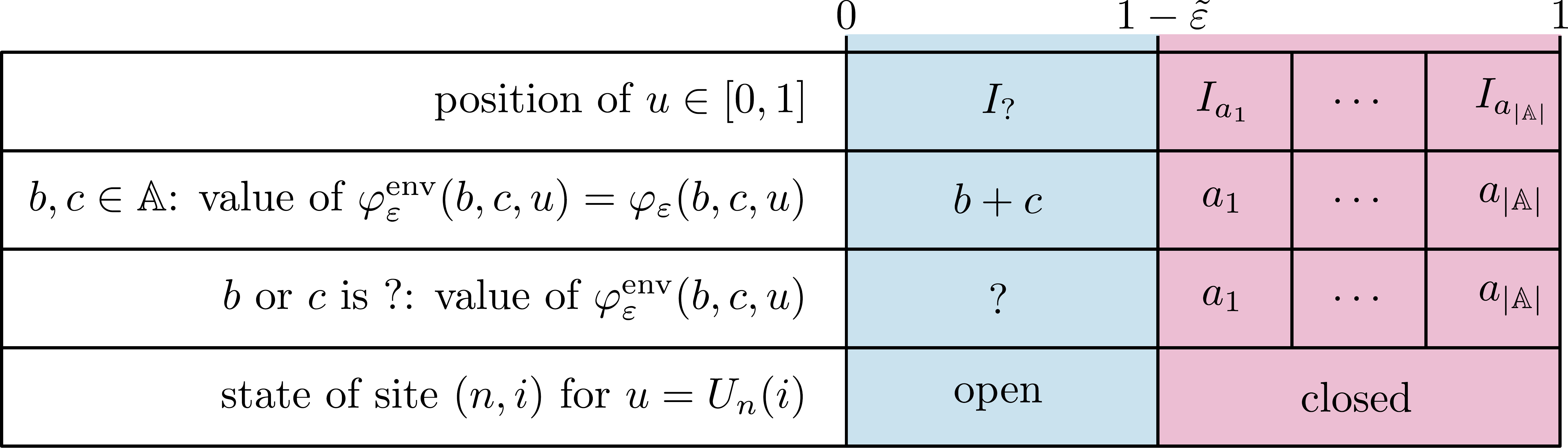}
\caption{Values of the updating functions and coupling with the percolation model}
\label{fig:updating}
\end{figure} 

\subsubsection{Coupling with the percolation model} 
\label{sec:coupling_with_perco}
We come back now to the oriented site percolation model we have defined in Section~\ref{sec:perco} from the \emph{same} random variables $\bigl(U_n(i)\bigr)_{n,i\in \ZZ}$. We just set the parameter $p$ to be equal to $1-\epst$. In this way, the site $(n,i)$ is open if and only if $U_n(i)\in I_\qm $ (see Figure~\ref{fig:updating}). 

The way we constructed the three processes together yields interesting properties which we describe in the two following lemmas. 

\begin{lemma}
\label{lemma:couplage_perco}
	For each $n\ge n_0+1$, $i\in\ZZ$,  we have $\Xenv_n(i)=\qm $ if and only if in the percolation model there exists $j\in\ZZ$ such that $(n_0+1,j)$ leads to $(n,i)$. 
\end{lemma}
\begin{proof}
We prove the lemma by induction on $n\ge n_0+1$. Since $X_{n_0}^{\text{env}}(i)=\qm $, we have $\Xenv_{n_0+1}(i)=\qm $ if and only if $U_{n_0+1}(i)\in I_\qm $, that is if and only if the site $(n_0+1,i)$ is open, which proves the result for $n=n_0+1$. Now assume that the result is true up to level $n\ge n_0+1$. By construction, $\Xenv_{n+1}(i)=\qm $ if and only if the two following conditions hold
 \begin{itemize}
  \item $\Xenv_{n}(i)=\qm $ or $\Xenv_{n}(i+1)=\qm $, which by assumption means that there exists $j\in\ZZ$ such that $(n_0+1,j)$ leads to $(n,i)$ or to $(n,i+1)$;
  \item $U_{n+1}(i)\in I_\qm $, which means that the site $(n+1,i)$ is open.
 \end{itemize}
 Therefore $\Xenv_{n+1}(i)=\qm $ if and only if there exists $j\in\ZZ$ such that $(n_0+1,j)$ leads to $(n+1,i)$, which ends the proof.
\end{proof}

The next lemma is a key result: it shows that wherever the process $\Xenv$ displays symbols different from $\qm $, the process $X^x$ displays the same symbols as $\Xenv$, \emph{regardless of the initial configuration $x$}.

\begin{lemma}
\label{lemma:synchro}
	For each $n\ge n_0$, $i\in\ZZ$ and $x\in\AA^\ZZ$, if $\Xenv_{n}(i) \neq \qm$ then $X_n^{x}(i) = \Xenv_{n}(i)$. 
\end{lemma}

\begin{proof}
Again we prove the result by induction on $n\ge n_0$. For $n=n_0$ this is obvious since $\Xenv_{n_0}(i)=\qm $ for all $i\in\ZZ$. Now assume that the property holds up to some level $n\ge n_0$, and let us consider some site $(n+1,i)$. 
 There are two cases.  
 \begin{itemize}
  \item First case: $\Xenv_{n}(i)=\qm $ or $\Xenv_{n}(i+1)=\qm $. Then the only way to have $\Xenv_{n+1}(i) \neq \qm $ is that $U_{n+1}(i)\in I_a$ for some $a\in\AA$. By construction of the updating functions, this implies that $X^{x}_{n+1}(i)=a=\Xenv_{n+1}(i)$.
  \item Second case: $\Xenv_{n}(i)\in\AA$ and $\Xenv_{n}(i+1)\in\AA$. Then by assumption this implies that we have both $X^{x}_{n}(i)=\Xenv_{n}(i)$ and $X^{x}_{n}(i+1)=\Xenv_{n}(i+1)$. By construction of the updating functions, we deduce on the one hand that $\Xenv_{n+1}(i)\in\AA$, and on the other hand that $X^{x}_{n+1}(i)=\Xenv_{n+1}(i)$.
 \end{itemize}
 In both cases the property also holds for the site $(n+1,i)$.
\end{proof}

\subsubsection{Construction of the process $(X_n)$ generating $\fte$}
We are now going to let $n_0$ vary in the preceding construction, and to make the dependence on $n_0$ more explicit we now denote 
$(X^{n_0,x}_n)_{n\ge n_0}$ and $(X^{n_0, \text{env}}_n)_{n\ge n_0}$ the processes respectively initialized at the configuration $x$ and at the configuration $\ldots \qm \qm \qm \qm \ldots$, and defined inductively by~\eqref{eq:rec_te} and~\eqref{eq:rec_env}. 

Since the initial configuration of the envelope processes is $\sigma$-invariant,  the action of the shifts $\sigma_{n,i}$ on the envelope processes yields the following formula, valid for all $n_0\le n$, $i\in\ZZ$:
\begin{equation}
\label{eq:shift_env}
 X^{n_0, \text{env}}_n(i) = X^{n_0-n, \text{env}}_0(0)\circ\sigma_{n,i}.
\end{equation}

By a similar argument to the proof of Lemma~\eqref{lemma:synchro}, we get the following relations between all those processes.
\begin{lemma}
\label{lemma:synchrobis}
	For each $n\geq n_0$, $i\in\ZZ$, if $X^{n_0, \text{env}}_{n}(i) \neq \qm $, then for all $n_1\le n_0$ we have $X^{n_1, \text{env}}_{n}(i) = X^{n_0, \text{env}}_{n}(i)$.
\end{lemma}

The next proposition will be used to define the process $(X_n)_{n\in\ZZ}$, evolving according to $\te$, under the assumption that $\eps$ be not too close to 0. We recall $p_c$ stands for the critical value for the oriented site percolation process, defined in Theorem~\ref{thm:perco_critique}. 
\begin{prop}
 \label{prop:convergence}
 Assume that
\begin{equation}
 \label{eq:epsilon2}
 (1-p_c) \frac{\cA-1}{\cA} < \eps <\frac{\cA-1}{\cA},\quad \text{\textit{i.e.} }0<1-\epst<p_c. 
\end{equation} 
Then, with probability 1, for all $n,i\in\ZZ$,  there exists a random variable $N_0(n,i)\le n$, taking values in $\ZZ$, and satisfying the following properties:
\begin{itemize}
 \item $N_0(n,i)$ is measurable with respect to $\G_n=\tribu\bigl(U_m(j):\ m\le n,\ j\in\ZZ\bigr)$.
 \item For each integers $n_1\le n_0\le N_0(n,i)$, we have
 \[
  X_n^{n_1, \text{env}}(i) = X_n^{n_0, \text{env}}(i) \neq \qm.
 \]
 \item For each $n,i\in\ZZ$, $N_0(n,i)=N_0(0,0)\circ \sigma_{n,i}$.
\end{itemize}
\end{prop}

\begin{proof}
Recall that the right-hand-side inequality, $\eps < \frac{\cA-1}{\cA}$, is just the the reformulation of condition~\eqref{eq:epsilon}.
The left-hand-side inequality, $(1-p_c) \frac{\cA-1}{\cA} < \eps$ is equivalent to $p<p_c$, where $p=1-\epst=1-\frac{ \eps\cA}{\cA-1}$ is the parameter of the percolation process constructed in the coupling. 
 
 Consider first the case $(n,i)=(0,0)$. For $p<p_c$, Corollary~\ref{cor:perco} ensures that the random variable $N_0=N_0(0,0)=\min\{m\le0: O_m\text{ does not hold}\}$ is almost surely well defined and $\G_0$-measurable.  Then Lemma~\ref{lemma:couplage_perco} shows that for each $n_0\le N_0(0,0)$, $X_0^{n_0, \text{env}}(0) \neq \qm $, and Lemma~\ref{lemma:synchrobis} gives 
 \[n_1\le n_0\le N_0(0,0) \Longrightarrow  X_0^{n_1, \text{env}}(0) = X_0^{n_0, \text{env}}(0).\] 
 
Then we can generalize this result to any site $(n,i)\in \ZZ\times\ZZ$, setting $N_0(n,i):=N_0\circ \sigma_{n,i}$. Note that $N_0(n,i)$ is the greatest integer $n_0\le n$ such that there is no $j\in\ZZ$ with $(n_0,j)$ leading to $(n,i)$.
 \end{proof}

We assume now that $\eps$ satisfies~\eqref{eq:epsilon2}, and we explain how to construct the process $(X_n)$ satisfying the conditions (a), (b), (c), (d) and (e) (see page~\pageref{p:conditions}).  With the above proposition, we can almost surely define, for each $n,i\in\ZZ$, 
\begin{equation}
 \label{eq:def_pcate}
 X_n(i) := X_n^{N_0(n,i), \text{env}} (i) = \lim_{n_0\to -\infty} X_n^{n_0, \text{env}} (i).
\end{equation}

For each $n\le n_0\in\ZZ$, $X_n^{n_0, \text{env}}$ is measurable with respect to $\G_n=\tribu\left(U_m(j)\right)_{m\le n, j\in\ZZ}$, so the same holds for $X_n$ and we have condition~(d).

Considering an integer $n_0\le \min \{ N_0(n,i), N_0(n,i+1), N_0(n+1,i) \}$, we have simultaneously:
$X_n(i)=X_n^{n_0, \text{env}} (i)$, $X_n(i+1)=X_n^{n_0, \text{env}} (i+1)$ and 
$X_{n+1}(i)=X_{n+1}^{n_0, \text{env}} (i)$. As the process $\bigl(X_n^{n_0, \text{env}}\bigr)$ satisfies the induction relation~\eqref{eq:rec_env} at $(n+1,i)$, it is therefore the same for $\bigl(X_n\bigr)$. This proves that the process $\bigl(X_n\bigr)_{n\in\ZZ}$ actually evolves according to to the ACP $\te$, and we have conditions (c) and (e). 
Moreover, from~\eqref{eq:shift_env} we get for each $n,i\in\ZZ$
\begin{equation*}
\label{eq:shift_env2}
 X_n(i) = X_0(0)\circ\sigma_{n,i}.
\end{equation*}
This gives on the one side that each row $X_n$ is a $\sigma$-process, so we have condition (a), and on the other side we also get $X_n=X_0\circ \sigma_{n,0}$. The rows $X_n$ therefore all follow the same law, which has to be $\mu$ by Corollary~\ref{cor:te_ergodic}, and we have condition (b). This concludes the proof of Theorem~\ref{thm:fte_dyn_standard}.

\begin{remark}
 As already pointed out in the sketch of the proof, the argument relies on inequality~\eqref{eq:unif_minoration}, hence the result remains true for any probabilistic cellular automaton satisfying such inequality. In particular the method can be applied to any probabilistic cellular automaton on $\AA^\ZZ$ which is a random perturbation of a deterministic cellular automaton by addition of a random error $\xi$, provided the law of each $\xi_n(i)$ is bounded below by $\epst\,\U_\AA$. 
 
 We can even note that, in~\eqref{eq:unif_minoration}, the uniform measure $\U_\AA$ does not play a particular role and can be replaced by \emph{any} fixed probability measure on $\AA$ (we just have to adapt the length of the subintervals $(I_a)_{a\in\AA}$ accordingly). Therefore the argument also applies to any probabilistic cellular automaton $\gamma$ on $\AA^\ZZ$ for which the conditional law of $X_{n+1}(i)$ is of the form
 \[
  \PP\bigl(X_{n+1}(i)=a \,|\, \tribu(X_m:m\le n)\bigr) = f_\gamma\bigl(a\,|\,X_n(i),X_n(i+1)\bigr),
 \]
 and the local transition rules $f_\gamma$ satisfy
 \[
  \forall a,b,c\in\AA,\ f_\gamma(a\,|\,bc) \ge \epst\rho(a),
 \]
for some fixed probability measure $\rho$ on $\AA$ (here no algebraic assumption on $\AA$ is needed).
\end{remark}

\subsubsection{Questions about the factor filtration $(\fte,\sigma)$}

When $\eps$ satisfies~\eqref{eq:epsilon2}, we prove the dynamical standardness of $(\fte,\sigma)$ by immersing it into a factor filtration $(\G,\sigma)$ dynamically of product type. The way we construct $\G$, it is clear that it carries more information than $\fte$, so that $\fte$ is a strict subfiltration of $\G$. Thus we may ask whether, in this case, $(\fte,\sigma)$ is itself a factor filtration dynamically of product type.

\medskip

Another natural question is of course whether the dynamical standardness of $(\fte,\sigma)$ persists as $\eps$ gets close to 0, which prevents the percolation argument to apply. One could already ask whether, for $\eps$ close to 0, the criterion of dynamical I-cosiness is satisfied for the factor filtration $(\fte,\sigma)$. We do not have the answer to this question, but we provide below a somewhat naive argument suggesting that $(\fte,\sigma)$ might not be dynamically I-cosy when $\eps$ is close enough to 0. 

We consider the case where the group $\AA$ is $\ZZ/2\ZZ$. In an attempt to establish dynamical I-cosiness, we try to realize a real-time dynamical coupling of $(\fte,\sigma)$, independent in the distant past, and for which the two copies $X'$ and $X''$ of the process $X=(X_n)$ satisfy $X'_0(0)=X''_0(0)$ with probability close to 1. We want to construct this coupling in a dynamical system $(\Omega,\PP,T)$, in which for some $n_0<0$ we already have two independent copies $(X'_n)_{n\le n_0}$ and $(X''_n)_{n\le n_0}$ of $(X_n)_{n\le n_0}$ (of course, each $X'_n$ and each $X''_n$ is supposed to be a $T$-process of law $\mu$). We assume that we also have an independent family $(U_n)_{n\ge n_0+1}=\bigl(U_n(i):i\in\ZZ\bigr)_{n\ge n_0+1}$ of $T$-processes, each $U_n$ being of law $\lambda^{\otimes \ZZ}$ ($\lambda$ is the Lebesgue measure on $[0,1]$), and the $U_n$'s being independent. We want to use this family to construct error processes $(\xi'_n)_{n\ge n_0+1}$ and $(\xi''_n)_{n\ge n_0+1}$, and then complete the processes $X'$ and $X''$ inductively via the usual relations 
\[ X'_{n}=\tau X'_{n-1}+\xi'_n\text{ and }X''_{n}=\tau X''_{n-1}+\xi''_n\quad(n\ge n_0+1). \]
As in Section~\ref{sec:fte_standard}, we still want to define the error processes in the form 
$\xi'_n(i)=g'_{n,i}\bigl(U_n(i)\bigr)$ and $\xi''_n(i)=g''_{n,i}\bigl(U_n(i)\bigr)$, where $g'_{n,i}$ and $g''_{n,i}$ are random updating functions, but we point out the essential difference with the construction in Section~\ref{sec:fte_standard}: we now want the law of the coupling to be invariant with respect to the left shift of coordinates. This means that these updating function are still allowed to depend on the realizations of the processes up to time $n-1$, but in an equivariant way. In other words, they can be chosen according to what we see around the site $(n,i)$, but the rule for the choice should be independent of $i$.  

We also introduce the auxiliary process $Z:=X'-X''$, which satisfies the same relation~\eqref{eq:evolution_Zn} for $n\ge n_0+1$. To get dynamical I-cosiness, we have to realize the coupling in such a way that 
\[ \PP\bigl(Z_0(i)=0\bigr)\tend{n_0}{-\infty}1 \] 
(note that the probability on the left-hand side does not depend on $i$).

Here is a natural idea to get what we are aiming at: we wish to encourage as much as possible the appearance of $0$'s in the $Z$ process. For this, let us observe that we have for each $n,i\in\ZZ$
 \begin{align}
     \label{eq:minmax}
     \begin{split}
     0 \leq \PP\bigl({Z}_{n}(i)  \neq \tau {Z}_{n-1}(i)\bigr)  
       & =  \PP\bigl(\xi'_n(i)  \neq \xi''_n(i)\bigr) \\
       & \leq \PP\bigl(\xi'_n(i)  \neq 0 \text{ or } \xi''_n(i)\neq 0\bigr) \\
       & \leq \PP\bigl(\xi'_n(i)  \neq 0\bigr) + \PP\bigl(\xi''_n(i)\neq 0\bigr) \\       
       & \leq 2 \eps.
	\end{split}
	\end{align}
Moreover, we have the possibility to make the inequality on the left-hand side an equality, by deciding to take $g'_{n,i}=g''_{n,i}$ (we call it \emph{Option~0}). On the other hand, we can also choose to turn the right-hand side inequalities into equalities, by choosing $g'_{n,i}$ and $g''_{n,i}$ in such a way that the subsets $\{u\in[0,1]:g'_{n,i}(u)=1\}$ and $\{u\in[0,1]:g''_{n,i}(u)=1\}$ are disjoint (we call it \emph{Option~1}). So a natural dynamical coupling is obtained by the following strategy: knowing the processes up to time $n-1$, we compute $\tau Z_{n-1}(i)$. If $\tau Z_{n-1}(i)=0$ we decide to keep this 0, by choosing Option~0. If $\tau Z_{n-1}(i)=1$, we maximize the probability that ${Z}_{n}(i)$ will still be 0 by choosing Option~1.

Applying systematically this strategy, we are left with a process $(Z_n)_{n\ge n_0}$ taking values in $\ZZ/2\ZZ$,  which evolves according to the following PCA rule: for each $n,i$, if $\tau Z_{n-1}(i)=0$ then $Z_n(i)=0$, otherwise $Z_n(i)=0$ with probability $2\eps$. Observe that the initial configuration $Z_{n_0}$ has law $\U_{\ZZ/2\ZZ}^{\otimes\ZZ}$, so that the initial density of 1's is $1/2$. The question now is whether those 1's can survive over the long term?

It turns out that this type of evolution falls into the domain of a family of PCA's studied by Bramson and Neuhauser~\cite{bramson1994}, who have proved that if $\eps$ is small enough, then symbols 1 survive. More precisely, we can state the following consequence of their result:
\begin{theo}[Bramson-Neuhauser]
	In the coupling described above, for $\eps>0$ small enough, there exists $\rho=\rho(\eps)>0$ such that 
	\[ \liminf_{n_0\to-\infty} \PP\bigl(Z_0(0)=1\bigr) \geq \rho. \]
\end{theo}

This shows that this attempt to get dynamical I-cosiness for $(\fte,\sigma)$ fails if $\eps$ is small. 

\subsection{A partial converse to Theorem~\ref{thm:dstd_implies_dic}}
\label{sec:converse}

The question whether dynamical I-cosiness of a factor filtration implies its dynamical standardness (as in the static case) is so far open. However the construction given in Section~\ref{sec:fte_dyn_standard} (for $\eps$ not too close to 0) provides an example where we are able to derive dynamical standardness from a particular form of dynamical I-cosiness. We want now to formalize this kind of situation, which can be viewed as the dynamic counterpart of the so-called \emph{Rosenblatt's self-joining criterion} in~\cite{laurent2011}.

In a dynamical system $(\Omega,\PP,T)$, consider a factor filtration $(\F,T)$ generated by a family $X=(X_n)_{n\le0}$ of  $T$-processes, each $X_n(i)$ taking its values in a fixed finite set $\AA$. We also assume the existence of a factor filtration $(\G,T)$, dynamically of product type, and generated by 
a family of independent $T$-processes $(U_n)_{n\le0}$ satisfying for each $n\le0$:
\begin{itemize}
 \item $U_n$ is independent of $\F_{n-1}\vee\G_{n-1}$,
 \item $\F_n\subset\F_{n-1}\vee\tribu(U_n)$.
\end{itemize}
(The two conditions above mean that $(U_n)_{n\le 0}$ is a \emph{superinnovation} of $\F$, according to Definition~3.10 in~\cite{laurent2011}.)
In particular there exists a measurable map $\phi_n$ such that 
\[ X_n = \phi_n(U_n,X_{n-1},X_{n-2},\ldots). \]
We can also assume without loss of generality that there exists in $(\Omega,\PP,T)$ a copy $X'$ of $X$, whose coordinates $X'_n$ are also $T$-processes, and which is independent of $\G_0\vee\F_0$ (if not, we place ourselves on the product $(\Omega\times\Omega,\PP\otimes\PP,T\times T)$ as in the proof of Lemma~\ref{lemma:dpt_implies_dic}). Now we build a family of dynamical real-time couplings of $(\F,T)$ in the following way: for a fixed $n_0\le0$, define $\bar X^{n_0}={(\bar X_n^{n_0})}_{n\le0}$ by
\begin{itemize}
 \item $\bar X_n^{n_0}=X'_n$ for $n\le n_0$,
 \item then, inductively for $n_0+1\le n\le 0$: $\bar X_n^{n_0}:=\phi_n(U_n,\bar X_{n-1}^{n_0},\bar X_{n-2}^{n_0},\ldots)$.
\end{itemize}
(The important point here is that, for $n\ge n_0+1$, the \emph{same} random variable $U_n$ is involved in the inductive relations giving $X_n$ and $\bar X_n^{n_0}$.) Then the filtration $\bar \F^{n_0}$ generated by $\bar X^{n_0}$ is dynamically isomorphic to $\F$, and the pair $\bigl((\F,T),(\bar\F^{n_0},T\bigr)$ is a dynamical real-time coupling of $(\F,T)$, which is $n_0$ independent.

\begin{definition}[Simple dynamical I-cosiness]
\label{def:sdic}
	If, in the above construction, we have for each integer $m\le0$
	\begin{equation}
        \label{eq:sdc}\PP\Bigl(\bar X^{n_0}_m(0) = X_m(0)\Bigr)\tend{n_0}{-\infty} 1,
	\end{equation}
	then we say that the factor filtration $(\F,T)$ is \emph{simply dynamically I-cosy}.
\end{definition}

Clearly, simple dynamical I-cosiness implies dynamical I-cosiness. Observe also that, when $\eps$ satisfies~\eqref{eq:epsilon2}, the results of Section~\ref{sec:fte_dyn_standard} prove that the factor filtration $(\fte,\sigma)$ is simply dynamically I-cosy.

The following theorem generalizes the conclusion of Section~\ref{sec:fte_dyn_standard}. 

\begin{theo} 
\label{thm:sdic}
If  the factor filtration $(\F,T)$ is simply dynamically I-cosy, then it is dynamically standard.
\end{theo}

\begin{proof}
We fix some integer $m\le0$. For each $n_0\le m-1$, we denote by $U_{n_0+1}^m$ the collection of $T$-processes $U_{n_0+1},\ldots,U_m$. We have
\begin{align*}
  \PP\Bigl(\bar X^{n_0}_m(0) = X_m(0)\Bigr)
            & = \sum_{a\in\AA} \PP\Bigl(\bar X^{n_0}_m(0) = X_m(0) = a\Bigr) \\
            & = \sum_{a\in\AA} \EE_\PP\left[ \PP\Bigl(\bar X^{n_0}_m(0) = X_m(0) = a \, | \, U_{n_0+1}^m\Bigr) \right].
\end{align*}
But $\bar X^{n_0}_m(0)$ and $X_m(0)$ are independent conditionally to $U_{n_0+1}^m$: indeed, once we know $U_{n_0+1}^m$, $\bar X^{n_0}_m(0)$ only depends on $\bar X_{n_0}$ while $X_m(0)$ only depends on $X_{n_0}$. We thus get
\begin{align*}
  \PP\Bigl(\bar X^{n_0}_m(0) = X_m(0)\Bigr)
            & = \sum_{a\in\AA} \EE_\PP\left[ 
            \PP\Bigl(\bar X^{n_0}_m(0) = a\, | \, U_{n_0+1}^m\Bigr) 
            \PP\Bigl(X_m(0) = a \, | \, U_{n_0+1}^m\Bigr) \right]\\
            & = \sum_{a\in\AA} \EE_\PP\left[
            \PP\Bigl(X_m(0) = a \, | \, U_{n_0+1}^m\Bigr)^2 \right]\\
            & =\EE_\PP\left[ \sum_{a\in\AA} 
            \PP\Bigl(X_m(0) = a \, | \, U_{n_0+1}^m\Bigr)^2 \right].
\end{align*}
Let $M_{n_0}$ be the random variable, measurable with respect to $\tribu(U_{n_0+1}^m)$, and defined by 
\[ M_{n_0} := \arg\max_{a\in\AA} \PP\Bigl(X_m(0) = a \, | \, U_{n_0+1}^m\Bigr). \]
(In case the maximum is reached for several symbols in $\AA$, we take the first of them with respect to some previously chosen total order on $\AA$.) We get
 \begin{align*}
  \PP\Bigl(\bar X^{n_0}_m(0) = X_m(0)\Bigr)
            & \le \EE_\PP\left[ \sum_{a\in\AA} 
            \PP\Bigl(X_m(0) = M_{n_0} \, | \, U_{n_0+1}^m\Bigr)\PP\Bigl(X_m(0) = a \, | \, U_{n_0+1}^m\Bigr) \right]\\
            & = \EE_\PP\Biggl[ \PP\Bigl(X_m(0) = M_{n_0} \, | \, U_{n_0+1}^m\Bigr) \underbrace{\sum_{a\in\AA} 
            \PP\Bigl(X_m(0) = a \, | \, U_{n_0+1}^m\Bigr)}_{=1} \Biggr]\\
            & = \PP\Bigl(X_m(0) = M_{n_0} \Bigr).
\end{align*}
But we assumed that $(\F,T)$ is simply dynamically I-cosy. So we have~\eqref{eq:sdc}, which yields
\[ \PP\Bigl(X_m(0) = M_{n_0} \Bigr) \tend{n_0}{-\infty} 1. \] 
For each integer $k\ge1$, we can thus find $n_k\le m-1$ such that 
\[ \PP\Bigl(X_m(0) \neq M_{n_k} \Bigr) < 2^{-k}. \] 
By Borel-Cantelli, we get $X_m(0)=\lim_{k\to\infty}M_{n_k}$ ($\PP$-a.s.). Moreover, $M_{n_k}$ is always $\G_m$-measurable. Therefore $X_m(0)$ is measurable with respect to $\G_{m}$, and since the sigma-algebras are factor sigma-algebras, the whole $T$-process $X_m$ is $\G_{m}$-measurable. 

This already proves that $\F$ is a subfiltration of the natural filtration $\G$ of the process $(U_n)_{n\le 0}$, hence $(U_n)_{n\le 0}$ is a \emph{generating} superinnovation of $\F$. By proposition~3.11 in~\cite{laurent2011}, it follows that $\F$ is immersed in $\G$. Since the factor filtration $(\G,T)$ is dynamically of product type, we get the dynamical standardity of $(\F,T)$.

\end{proof}

\section*{Acknowledgements}

The authors wish to thank Irène Marcovici for fruitful conversations which inspired the subject of this paper, Séverin Benzoni, Christophe Leuridan, Alejandro Maass  and an anonymous referee for their careful reading, relevant questions and valuable comments.

\bibliography{filtration-CA}

\end{document}